\documentclass[a4paper,12pt]{article}
\usepackage[utf8]{inputenc}
\usepackage{graphics} 
\usepackage[dvips,pdftex]{graphicx} 
\usepackage[usenames,dvipsnames]{color} 
\usepackage[T1]{fontenc} 
\usepackage{amsfonts} 
\usepackage{enumerate} 
\usepackage{amsthm}
\usepackage{amsmath} 
\usepackage[pdftex]{hyperref}  
\hypersetup{colorlinks=true,linkcolor=red,citecolor=green} 

\newtheorem{Pro}{Proposition}[section]
\newtheorem{Teo}{Theorem}[section]
\newtheorem{Lema}{Lemma}[section]
\newtheorem{Cor}{Corollary}[section]
\newtheorem{remark}{Remark}
\newtheorem{Def}{Definition}

\newenvironment{demThs}{\textit{Proof of Theorem~\ref{Thsistema}.}}{\hfill$\square$}

\usepackage[tmargin=3cm,bmargin=2cm,lmargin=3cm,rmargin=2cm]{geometry}
\definecolor{cinza}{gray}{.8}
\definecolor{branco}{gray}{1}
\definecolor{preto}{gray}{0}
\definecolor{verdemusgo}{rgb}{.3,.7,.5}
\definecolor{vinho}{cmyk}{0,1,1,.5}
\newcommand{\R}{\mathbb{R}}

\usepackage{graphicx}

\title{Unilateral global bifurcation for a class of quasilinear elliptic systems and applications\footnote{CMR and AS have been partially supported for the project MTM2015-69875-P (MINECO/FEDER, UE) and AS by the project CNPQ-Proc. 400426/2013-7. }}
\author{W. Cintra$^1$, C. Morales-Rodrigo$^2$ and A. Suárez$^2$}

\begin{document}

\maketitle
\begin{center}
1. Departamento de Matemática, Universidade de Bras\'ilia, \\
70910-900, Brasília - DF, Brazil \\
2. Dpto. de Ecuaciones Diferenciales y Análisis Numérico \\
Fac. de Matemáticas, Univ. de Sevilla\\
Calle Tarfia s/n - Sevilla Spain

E-mail addresses: willian@unb.br, cristianm@us.es, suarez@us.es
\end{center}

\begin{abstract}
In this paper we establish a unilateral bifurcation result for a class of quasilinear elliptic system strongly coupled, extending the bifurcation theorem of \cite{LGapela}. To this aim, we use several results, such that Fredholm operator of index $0$ theory, eigenvalues of elliptic operators and the Krein-Rutman theorem. Lastly, we apply this result to some particular systems arising from population dynamics and determine a region of existence of coexistence states. 
\end{abstract}

\textbf{Keywords:} quasilinear elliptic system; global bifurcation; coexistence states;

\textbf{Mathematics Subject Classification (2010):} 
35B32, 
35J57, 
35J62 
47J15, 
47A13, 
92D25. 
\section{Introduction}\label{sec1}
In this paper we will study via bifurcation theory the following quasilinear elliptic system
\begin{equation}\label{sistema}
    \left\{ \begin{array}{ll}
         -\mbox{div}(P(u,v) \nabla u + S(u,v) \nabla v) =\lambda a(x)u + f(x,u)u+ F(x,u,v)uv & \mbox{in}~\Omega,\\
         -\mbox{div}(Q(u,v) \nabla u + R(u,v) \nabla v) = \mu b(x)v + g(x,v)v+ G(x,u,v)uv &\mbox{in}~\Omega,\\
         u=v=0 & \mbox{on}~\partial \Omega,
    \end{array} \right.
\end{equation}
where $\Omega \subset \R^N$, $N \geq 1$ is a bounded domain with a smooth boundary. We assume  
\begin{enumerate}
    \item[($H_{PQRS}$)] $P(u,v)$, $R(u,v)$, $Q(u,v)$ and $S(u,v)$ are real functions defined in $[0,+\infty)\times [0,+\infty)$ of class $\mathcal{C}^2$ such that:
    \begin{equation}\label{hipq}
        Q(u,0)=0 \quad      \forall u \geq 0,
    \end{equation}            
    \begin{equation}
        S(0,v) = 0 \quad      \forall v \geq 0,
    \end{equation}
    \begin{equation}\label{determinante}
        |P(u,v)R(u,v)-Q(u,v)S(u,v)| \geq \delta_0>0 \quad      \forall u,v\geq 0  
    \end{equation}
    and
    \begin{equation}\label{PR}
        P(u,v)\geq P_0 >0, \quad R(u,v)\geq R_0>0 \quad      \forall u,v\geq 0, 
    \end{equation}
    where $\delta_0$, $P_0$ and $R_0$ are positive constants.
    \item[($H_{ab}$)] $a,b:\overline{\Omega} \rightarrow [0,\infty)$ are continuous, non-negative and nontrivial functions.
\end{enumerate}
With respect to the reaction terms, we consider $\lambda,\mu\in\R$ as bifurcation parameters, and we also assume:
\begin{enumerate}
    \item[($H_{fg}$)] $f(x,w)$ and $g(x,w)$ are real functions defined in  $\overline{\Omega} \times \R$, continuous in  $x$ and of class $\mathcal{C}^1$ in $w$ such that
    $$f(x,0)=g(x,0)=0 \quad \forall x \in \overline{\Omega};$$
    \item[($H_{FG}$)] $F(x,u,v)$ and $G(x,u,v)$ are real functions defined in  $\overline{\Omega}\times \R^2$, continuous in $x$ and of class $\mathcal{C}^1$ in $(u,v)$.
\end{enumerate}
Observe that our hypotheses are similar to those in \cite{Bifbook} (see Section 7.2).

System (\ref{sistema})  is of particular interest from the point of view of the applications and it arises in many important problems, such that reaction-diffusion models in population dynamics (see \cite{Cosner,CC2003,SKT}) and the Keller-Segel models (see \cite{chemotaxis, SK2003}). For instance, in population dynamics the insertion of nonlinear diffusion terms in the left side of (\ref{sistema}) describes more realistic situation if compared with semilinear case. In this context, the functions $Q$,$S$ and $P$,$R$ are often called  cross-diffusion and self-diffusion terms, respectively.

From the mathematical point of view, there exists an extensive list of articles that deal with the following particular case of (\ref{sistema})
\begin{equation}\label{ss2}
    \left\{ \begin{array}{ll}
         -\Delta(\varphi(u,v)) =\lambda a(x)u + f(x,u)u+ F(x,u,v)uv & \mbox{in}~\Omega,\\
         -\Delta(\psi(u,v)) = \mu b(x)v + g(x,v)v+ G(x,u,v)uv & \mbox{in}~\Omega,
    \end{array} \right.
\end{equation}
with some boundary conditions, specially motivated by the well-known model  proposed by  Shigesada, Kawasaki and Teramoto \cite{SKT}. Moreover, several techniques have been used to deal with (\ref{ss2}), for instance fixed point index \cite{Ruan,ruan1999,ryu,LouNi,JunKim2014}, sub-supersolution methods \cite{pao2005} and bifurcation theory \cite{DuYuan,DMS2009,KatoY2004,yamada96,yamada2004,wang2015, Kuto2009} and references therein.

However, for  general systems (\ref{sistema}) there are fewer results available in the literature. We mention  \cite{pontofixo} that applies fixed point methods for the case in that $P,Q,R,S$ are suitable polynomial function on $u$ and $v$, which is not required here, and the interesting bifurcation results developed by \cite{shi} which have been widely used in recent years (see, for instance, \cite{WangYanGai, wang2015,memoria2016,ChenTongWang}). We point out that the main theorem of \cite{shi} is based on Degree theory for $\mathcal{C}^1$ Fredholm mappings of index $0$ of \cite{Pejsachowicz} and the unilateral bifurcation results of \cite{Bifbook} (see the proof of Theorem 4.4 in \cite{shi}). The result provides  sufficient conditions to obtain a unilateral global bifurcation result of positive solutions of a class of quasilinear elliptic systems. Our results  are also based on  \cite{Bifbook} (see also \cite{LG-MC2005}), but with a different approach. Precisely, we will extend the bifurcation results for semilinear system of \cite{LGapela}  to the quasilinear case (\ref{sistema}), which requires the existence of a non-degenerate semitrivial positive solution (in a sense that will be defined later) to guarantee the existence of global bifurcation of non-negative solutions and provides better information on the continuum obtained. 

We also refer to \cite{WCS4} where the authors developed a  bifurcation result to analyze a predator-prey system which is a particular case of (\ref{sistema}). Actually, here we extend this result to the general system (\ref{sistema}).

It should be noted that a standard approach to deal with (\ref{ss2}) is to apply the change of variables $U=\varphi(u,v)$ and $V=\psi(u,v)$, transforming (\ref{ss2}) into a semilinear  system, decoupled in the diffusion, where one can apply the techniques which worked for such systems, for instance, the bifurcation theorems of \cite{LGapela}. However, (\ref{sistema}) can not be written as (\ref{ss2}) because there does not exist an immediate  change variable that transforms (\ref{sistema}) into a semilinear system. 



In order to state our main result, we need to make some considerations. First, we search for non-negative strong solutions of (\ref{sistema}) (in $W_0^{2,p}(\Omega)\times W_0^{2,p}(\Omega)$, for all $p>1$). Observe that (\ref{sistema}) admits three types of non-negative strong solutions: the trivial solution $(0,0)$; the semitrivial positive solutions $(u,0)$ and $(0,v)$ where $u$ and $v$ are positive solutions of 
\begin{equation}\label{stl}
    \left\{ \begin{array}{ll}
    -\mbox{div}(P(u,0) \nabla u) = \lambda a(x) u + f(x,u)u     &\mbox{in}~\Omega,  \\
    u=0     & \mbox{on}~\partial\Omega,
    \end{array}\right.
\end{equation}
and
\begin{equation}\label{stm}
    \left\{ \begin{array}{ll}
    -\mbox{div}(R(0,v) \nabla v) = \mu b(x)v + g(x,v)v     &\mbox{in}~\Omega,  \\
    v=0     & \mbox{on}~\partial\Omega,
    \end{array}\right.
\end{equation}
respectively; and the \textit{coexistence states} $(u,v)$ with both components positive. With respect to the positive semitrivial solutions we also need the following definition:
\begin{Def}\label{degenerada}
Let $(\lambda,\theta_\lambda)$ be a non-negative solution of (\ref{stl}). $(\lambda,\theta_\lambda)$  is a non-degenerate solution of (\ref{stl}) if zero is the unique strong solution of the linearization of  (\ref{stl}) at $\theta_\lambda$, which is given by
\begin{equation}\label{ndegen}
\left\{    \begin{array}{ll}
         -\mbox{div}(P_u(\theta_\lambda,0) u \nabla \theta_\lambda + P(\theta_\lambda,0) \nabla u)  = \lambda a(x)u +[\theta_\lambda  f_u(x,\theta_\lambda)+f(x,\theta_\lambda)] u&\mbox{in}~\Omega,  \\
         u=0&\mbox{on}~\Omega.
    \end{array}\right.
\end{equation}
\end{Def}
In an analogous way, the non-degenerate solution of (\ref{stm}) is defined.

Now, we will introduce some notations with respect to an important eigenvalue problem. Let ${\cal P}$ denote the positive cone in $\mathcal{C}_0^1(\overline{\Omega})$ whose  interior (notation: $\mbox{int}({\cal P})$) is nonempty. 

On the other hand, given $A \in \mathcal{C}^1(\overline{\Omega})$ and $B,C\in \mathcal{C}(\overline{\Omega})$, satisfying
$C(x) >0$, $A(x) \geq A_0>0$ for $x\in \Omega$ and a suitable constant $A_{0}$, we denote by
$$
\sigma_1[-\mbox{div} (A\nabla ) + B(x); C(x)]
$$
the principal eigenvalue  of
\begin{equation*}\label{autovalor}
\left\{\begin{array}{ll}
    -\mbox{div} \left(A(x) \nabla u\right) + B(x)u =\lambda C(x)u &\mbox{in}~\Omega,  \\
     u=0&\mbox{on}~\partial\Omega.
\end{array}  \right.
\end{equation*}
It is well-known (see for instance \cite{LGauto}) that $\sigma_1[-\mbox{div} (A\nabla ) + B(x); C(x)]$ is increasing with respect to $A$ 
and $B$. Moreover,  it is decreasing (resp. increasing) with respect to $C$ if $\sigma_1[-\mbox{div} (A\nabla ) + B(x); 1]>0$ (resp.$\sigma_1[-\mbox{div} (A\nabla ) + B(x); 1]<0$. See, for instance, \cite{SMPbook,Djairo}.

Now, let  $B,C \in \mathcal{C}(\overline{\Omega})$, $M_1,M_2 \in \mathcal{C}^2(\R)$, $M_0$ constant and $v \in \mathcal{C}^2(\overline{\Omega}) \cap \mbox{int}({\cal P})$ such that $M_2 \geq M_0>0$, $C>0$. Consider the following eigenvalue problem
\begin{eqnarray}\label{autoaux0}
\left\{ \begin{array}{ll}
-\mbox{div} (M_1(v)u\nabla v + M_2(v)\nabla u) + B(x) u = \lambda C(x) u     &\mbox{in}~ \Omega,  \\
u=0     &\mbox{on}~ \partial \Omega,
\end{array}   \right.
\end{eqnarray}
whose principal eigenvalue will be denoted by 
$$\sigma_1[-\mbox{div} (M_1(v)\nabla v + M_2(v)\nabla ) + B(x); C(x)].$$
On the other hand,  let $\varphi$ be a positive eigenfunction of  (\ref{autoaux0}) with $\|\varphi\|_0=1$. Using the change of variable
$$z = \varphi e^{h(v)} \Leftrightarrow  z e^{-h(v)}= \varphi, \quad h(v):=\int_0^v\frac{M_1(s)}{M_2(s)}ds$$
in (\ref{autoaux0}) with $(\lambda,u)=(\lambda_0,\varphi)$, where $\lambda_0 = \sigma_1[-\mbox{div} (M_1(v)\nabla v + M_2(v)\nabla ) + B(x); C(x)]$, we obtain
\begin{eqnarray*}
\left\{ \begin{array}{ll}
-\mbox{div} (M_2(v)e^{-h(v)}\nabla z) + B(x) ze^{-h(v)} = \lambda_0 C(x) z e^{-h(v)}     &\mbox{in}~ \Omega,  \\
z=0     &\mbox{on}~ \partial \Omega.
\end{array}   \right.
\end{eqnarray*}
Since $z = \varphi e^{-h(v)} >0$, it follows that 
$$\lambda_0 =  \sigma_1[-\mbox{div} (M_2(v)e^{-h(v)}\nabla) + B(x) e^{-h(v)};C(x) e^{-h(v)}].$$
Thus, 
\begin{multline}\label{av}
    \sigma_1[-\mbox{div} (M_1(v)\nabla v + M_2(v)\nabla ) + B(x); C(x)] = \cr \sigma_1[-\mbox{div} (M_2(v)e^{-h(v)}\nabla) + B(x) e^{-h(v)};C(x) e^{-h(v)}].
\end{multline}

Finally, to state our main result, we define
$$h_1,h_2:[0,\infty) \rightarrow \R$$
given by
\begin{equation}\label{h1}
    h_1(z):= \int_0^z\frac{Q_v(s,0)}{R(s,0)}ds
\end{equation}
and
\begin{equation}\label{h2}
    h_2(z):=\int_0^z\frac{S_u(0,s)}{P(0,s)}ds.
\end{equation}

Thus, we have

\begin{Teo}\label{Thsistema}
Suppose that  ($H_{PQRS}$), ($H_{ab}$), ($H_{fg}$) and ($H_{FG}$) are satisfied. Let  $(\lambda, \theta_\lambda) \in \R \times \mbox{int}({\cal P})$ be a nondegenerate solution of  (\ref{stl}) and consider
\begin{eqnarray}\label{mulambda}
    \mu_\lambda&:=&\sigma_1\left[-\mbox{div}\left(Q_v(\theta_\lambda,0)\nabla \theta_\lambda+R(\theta_\lambda,0) \nabla\right) - G(x,\theta_\lambda,0)\theta_\lambda;b \right] \nonumber\\
    &=& \sigma_1[-\mbox{div}( R(\theta_\lambda,0) e^{-h_1(\theta_\lambda)}\nabla) - G(x,\theta_\lambda,0)\theta_\lambda e^{-h_1(\theta_\lambda)}; be^{-h_1(\theta_\lambda)}].
\end{eqnarray}
Then, from the point  $(\mu,u,v)=(\mu_\lambda,\theta_\lambda,0)$ emanates a continuum
$$\mathfrak{C} \subset \mathbb{R}\times \mbox{int}({\cal P})\times\mbox{int}({\cal P})$$
of coexistence states of (\ref{sistema}) such that either: 
\begin{enumerate}[1.]
    \item $\mathfrak{C}$ is unbounded in $\R \times \mathcal{C}_0^1(\overline{\Omega})\times \mathcal{C}_0^1(\overline{\Omega})$; 
    \item There exists a positive solution $(\mu^*,\theta_{\mu^*})$ of (\ref{stm})
such that
\begin{eqnarray*}
\lambda &=&  \sigma_1\left[-\mbox{div}\left(S_u(0,\theta_{\mu^*})\nabla\theta_{\mu^*} + P(0,\theta_{\mu^*})\nabla\right) - F(x,0,\theta_{\mu^*})\theta_{\mu^*};a\right] \\
&=& \sigma_1[-\mbox{div}( P(0,\theta_{\mu^*}) e^{-h_2(\theta_{\mu^*})}\nabla) - F(x,\theta_{\mu^*},0)\theta_{\mu^*} e^{-h_2(\theta_{\mu^*})}; ae^{-h_2(\theta_{\mu^*})}]    
\end{eqnarray*}
and $(\mu^*,0,\theta_{\mu^*}) \in \overline{\mathfrak{C}}$; 
\item There exists another positive solution of (\ref{stl}), say $(\lambda,\psi_\lambda)$, with $\psi_\lambda \neq \theta_\lambda$ such that
\begin{eqnarray*}
(\sigma_1\left[-\mbox{div}\left(Q_v(\psi_\lambda,0)\nabla \psi_\lambda+R(\psi_\lambda,0) \nabla\right) - G(x,\psi_\lambda,0)\psi_\lambda;b \right], \psi_\lambda,0) =\\
(\sigma_1[-\mbox{div}( R(\psi_\lambda,0) e^{-h_1(\psi_\lambda)}\nabla) - G(x,\psi_\lambda,0)\psi_\lambda e^{-h_1(\psi_\lambda)}; be^{-h_1(\psi_\lambda)}],\psi_\lambda,0)
\in \overline{\mathfrak{C}};
\end{eqnarray*}
\item $\lambda = \sigma_1[-\mbox{div}(P(0,0)\nabla);a]$ and $(\sigma_1[-\mbox{div}(R(0,0)\nabla);b],0,0) \in \overline{\mathfrak{C}}.$
\end{enumerate}
\end{Teo}

In the same way, we can fix $\mu$ and consider $\lambda$ as a bifurcation parameter. Thus, we have:

\begin{Teo}\label{Thsistema'}
Suppose that ($H_{PQRS}$), ($H_{ab}$), ($H_{fg}$) end ($H_{FG}$) are satisfied. Let $(\mu, \theta_\mu) \in \R \times \mbox{int}({\cal P})$ be a nondegenerate solution of (\ref{stm}) and consider
\begin{eqnarray}\label{lambdamu}
\lambda_\mu &=&  \sigma_1\left[-\mbox{div}\left(S_u(0,\theta_{\mu})\nabla\theta_{\mu} + P(0,\theta_{\mu})\nabla\right) - F(x,0,\theta_{\mu})\theta_{\mu};a\right] \nonumber\\
&=&  \sigma_1[-\mbox{div}( P(0,\theta_{\mu}) e^{-h_2(\theta_{\mu})}\nabla) - F(x,\theta_{\mu},0)\theta_{\mu} e^{-h_2(\theta_{\mu})}; ae^{-h_2(\theta_\mu)}].
\end{eqnarray}
Then, from the point $(\lambda,u,v)=(\lambda_\mu,0,\theta_\mu)$ emanates a continuum 
$$\mathfrak{C} \subset \mathbb{R}\times \mbox{int}({\cal P})\times\mbox{int}({\cal P})$$
of coexistence states of (\ref{sistema}) such that either:
\begin{enumerate}[1.]
    \item $\mathfrak{C}$ is unbounded in $\R \times \mathcal{C}_0^1(\overline{\Omega})\times \mathcal{C}_0^1(\overline{\Omega})$; 
    \item There exists a positive solution $(\lambda^*,\theta_{\lambda^*})$ of (\ref{stl}) such that
\begin{eqnarray*}
\mu&=&\sigma_1\left[-\mbox{div}\left(Q_v(\theta_{\lambda^*},0)\nabla \theta_{\lambda^*}+R(\theta_{\lambda^*},0) \nabla\right) - G(x,\theta_{\lambda^*},0)\theta_{\lambda^*};b \right]\\
&=&  \sigma_1[-\mbox{div}( R(\theta_{\lambda^*},0) e^{-h_1(\theta_{\lambda^*})}\nabla) - G(x,\theta_{\lambda^*},0)\theta_{\lambda^*} e^{-h_1(\theta_{\lambda^*})}; be^{-h_1(\theta_{\lambda^*})}]
\end{eqnarray*}
and $(\lambda^*,\theta_{\lambda^*},0) \in \overline{\mathfrak{C}}$; 
\item There exists another positive solution of (\ref{stm}), say $(\mu,\psi_\mu)$, with $\psi_\mu \neq \theta_\mu$ such that
\begin{eqnarray*}
(\sigma_1\left[-\mbox{div}\left(S_u(0,\psi_\mu)\nabla\psi_\mu + P(0,\psi_\mu)\nabla\right) - F(x,0,\psi_\mu)\psi_\mu;a\right],0, \psi_\mu)= \\
 (\sigma_1[-\mbox{div}( P(0,\psi_\mu) e^{-h_2(\psi_\mu)}\nabla) - F(x,\psi_\mu,0)\psi_\mu e^{-h_2(\psi_\mu)}; ae^{-h_2(\psi_\mu)}],0, \psi_\mu) \in \overline{\mathfrak{C}};
\end{eqnarray*}
\item $\mu = \sigma_1[-\mbox{div}(R(0,0)\nabla);b]$ and $(\sigma_1[-\mbox{div}(P(0,0)\nabla);a],0,0) \in \overline{\mathfrak{C}}.$
\end{enumerate}
\end{Teo}

\begin{remark}\label{obs1}
An useful comment regarding the eigenvalues $\mu_\lambda$ and $\lambda_\mu$ in Theorems~\ref{Thsistema} and~\ref{Thsistema'} is necessary. Note that
\begin{equation}\label{au1}
      \mu_\lambda=\sigma_1\left[-\mbox{div}\left(Q_v(\theta_\lambda,0)\nabla \theta_\lambda+R(\theta_\lambda,0) \nabla\right) - G(x,\theta_\lambda,0)\theta_\lambda;b \right]
\end{equation}
or
\begin{equation}\label{au2}
\mu_\lambda=\sigma_1[-\mbox{div}( R(\theta_\lambda,0) e^{-h_1(\theta_\lambda)}\nabla) - G(x,\theta_\lambda,0)\theta_\lambda e^{-h_1(\theta_\lambda)}; be^{-h_1(\theta_\lambda)}],
\end{equation}
 Thanks to (\ref{hipq}) in ($H_{PQRS}$), (\ref{au1})  appears naturally when one linearizes the second equation of (\ref{sistema}) at $(\theta_\lambda,0)$ and we will be used throughout the proof of Theorem~\ref{Thsistema}. On the other hand,  (\ref{au2}) can be more convenient in applications, due to the monotonicity properties of this eigenvalue. The same remark applies to $\lambda_\mu$.
\end{remark}

\begin{remark}\label{obs}
Similar to what happens in Theorem~7.2.2 of \cite{Bifbook} (see Remark~7.2.3), the alternative  3 of Theorem~\ref{Thsistema} (resp. Theorem~\ref{Thsistema'}) cannot occur if (\ref{stl}) (resp. (\ref{stm})) has a unique positive solution and alternative 4 cannot occur if $\lambda \neq \sigma_1[-\mbox{div}(P(0,0)\nabla);a]$ (resp. $\mu \neq \sigma_1[-\mbox{div}(R(0,0)\nabla);b]$). This is a common situation in applications, as we will see in Section~\ref{sec5}.
\end{remark}

The paper is organized as follows. In Section~\ref{sec2}, we will re-write  (\ref{sistema}) as a suitable nonlinear equation to apply the unilateral bifurcation theorems of \cite{Bifbook}. In Section~\ref{sec3}  we will present some auxiliary results that will be useful to prove Theorem \ref{Thsistema}, which will be done in Section~\ref{sec4}. Finally, in Section \ref{sec5} we provide  some applications to systems arising to population dynamics.

\section{Construction of the operator}\label{sec2}
The main goal of this section is to re-write  (\ref{sistema}) as a suitable nonlinear equation to apply the unilateral bifurcation result of \cite{Bifbook}. To this end, we argue as follows.

\begin{remark} 
In \cite{Bifbook} (see Section 7.2) the author defines the operator on a subspace of $\mathcal{C}(\overline{\Omega})$, which is ordered Banach space whose positive cone is normal and has nonempty interior. As we shall see below, we can not do the same, due to the presence of the gradient of the functions $u$ and $v$. For this reason, the space chosen will be $\mathcal{C}_0^1(\overline{\Omega})$, which also is an ordered Banach space whose positive cone has nonempty interior, but it is not normal. However, we can still apply Theorem 12.3 and Corollary 12.4 of \cite{KreinR}.
\end{remark}

First, in view of (\ref{determinante}), we can extend the functions
$$P,Q,R,S:[0,\infty) \rightarrow \R$$
to $\R$ such that $P,Q,R,S \in \mathcal{C}^2(\R)$ and 
\begin{eqnarray}\label{determinante2}
|P(u,v)Q(u,v)-R(u,v)S(u,v)| \geq \widehat{\delta}_0  >0 \quad \forall u,v \in \R.
\end{eqnarray}
Thus, throughout the rest of this paper $P,Q,R,S:\R \rightarrow \R$ are function of class $\mathcal{C}^2$ satisfying ($H_{PQRS}$) and (\ref{determinante2}).

Suppose now that  $(u,v) \in  W_0^{2,p}(\Omega)\times W_0^{2,p}(\Omega)$, for all $p>1$, is a non-negative solution of (\ref{sistema}). Then (\ref{sistema})  is equivalent to
\begin{equation}\label{sistema'}
    \left\{ \begin{array}{ll}
         -P_u(u,v) |\nabla u|^2 -  S_v(u,v) |\nabla v|^2 - (P_v(u,v) + S_u(u,v))\nabla u  \nabla v & \\
         - P(u,v) \Delta u - S(u,v) \Delta v =\lambda a(x)u + f(x,u)u+ F(x,u,v)uv & \mbox{in}~\Omega,\\
         -Q_u(u,v) |\nabla u|^2 -  R_v(u,v) |\nabla v|^2 - (Q_v(u,v) + R_u(u,v))\nabla u  \nabla v & \\
         - Q(u,v) \Delta u - R(u,v) \Delta v =\mu b(x)v + g(x,v)v+ G(x,u,v)uv & \mbox{in}~\Omega,\\
         u=v=0 & \mbox{on}~\partial \Omega.
    \end{array} \right.
\end{equation}
Denoting by simplicity $P=P(u,v),~Q=Q(u,v),~R=R(u,v),~S=S(u,v) $ and the same for its derivatives, we can re-write (\ref{sistema'}) as 
\begin{equation*}
    \left\{ \begin{array}{ll}
         - P \Delta u - S \Delta v = M & \mbox{in}~\Omega,\\
         - Q \Delta u - R \Delta v = N & \mbox{in}~\Omega,\\
         u=v=0 & \mbox{on}~\partial \Omega,
    \end{array} \right.
\end{equation*}
where
$$M:= M(u,v)= P_u |\nabla u|^2 +  S_v |\nabla v|^2 + (P_v + S_u)\nabla u  \nabla v + \lambda a(x) u + f(x,u)u+ F(x,u,v)uv$$
and
$$N:= N(\mu,u,v) = Q_u |\nabla u|^2 +  R_v |\nabla v|^2 + (Q_v + R_u)\nabla u  \nabla v +\mu b(x)v + g(x,v)v+ G(x,u,v)uv.$$
Or in matrix form
\begin{equation}\label{okk}
    -\left[\begin{array}{cc}
        P & S \\
        Q & R 
    \end{array}\right] \left[\begin{array}{c}
        \Delta u \\
        \Delta v  
    \end{array}\right] = \left[\begin{array}{c}
        M \\
        N  
    \end{array}\right].
\end{equation}
By hypothesis (\ref{determinante2}), the matrix
\begin{equation*}
    \left[\begin{array}{cc}
        P(s,t) & S(s,t) \\
        Q(s,t) & R(s,t) 
    \end{array}\right]
\end{equation*}
is invertible for all $s,t \in \R$ and then (\ref{okk}) is equivalent to
\begin{equation*}
    -\left[\begin{array}{c}
        \Delta u \\
        \Delta v  
    \end{array}\right] =\frac{1}{PR-QS}
    \left[\begin{array}{cc}
        R & -S \\
        -Q & P 
    \end{array}\right] \left[\begin{array}{c}
        M \\
        N  
    \end{array}\right] = \left[\begin{array}{c}
        \displaystyle\frac{RM-SN}{PR-QS}  \\
        \noalign{\smallskip}
        \displaystyle\frac{PN-QM}{PR-QS}  
    \end{array}\right].
\end{equation*}
Let $(\lambda, \theta_\lambda)$  be a (positive) solution of  (\ref{stl}). Adding to both sides of the second equation above the following linear term
$$
Z(v):= \frac{-(Q_v(\theta_\lambda,0)+R_u(\theta_\lambda,0)) \nabla \theta_\lambda \nabla v + kv}{R(\theta_\lambda,0)}
$$
where $k>0$ is a constant to be chosen, it follows
\begin{equation*}
    \left[\begin{array}{c}
        -\Delta u \\
        -\Delta v +Z(v) 
    \end{array}\right] = \left[\begin{array}{c}
        \displaystyle\frac{RM-SN}{PR-QS}  \\
        \displaystyle\frac{PN-QM}{PR-QS} + Z(v)  
    \end{array}\right].
\end{equation*}
\begin{remark}
The addition of $Z$ is necessary for that the operator $T_\mu$ defined  in (\ref{T}) to be strongly positive (see Lemma~\ref{propT}). 
\end{remark}
Finally,
\begin{equation*}
    \left[\begin{array}{c}
         u \\
         v 
    \end{array}\right] = \left[\begin{array}{c}
        (-\Delta)^{-1}\left(\displaystyle\frac{RM-SN}{PR-QS}\right)  \\
        (-\Delta + Z)^{-1} \left(\displaystyle\frac{PN-QM}{PR-QS} + Z(v)\right)  
    \end{array}\right].
\end{equation*}
Thus, fixed  $\lambda \in \R$, we define the operator
$$\mathfrak{F}: \R \times \mathcal{C}_0^1(\overline{\Omega})\times \mathcal{C}_0^1(\overline{\Omega}) \longrightarrow  \mathcal{C}_0^1(\overline{\Omega})\times \mathcal{C}_0^1(\overline{\Omega}) $$
given by
\begin{equation}\label{F}
    \mathfrak{F}(\mu,u,v)  = \left[\begin{array}{c}
    u-    (-\Delta)^{-1}\left(\displaystyle\frac{RM-SN}{PR-QS} \right) (\mu,u,v)  \\
    v -   (-\Delta + Z)^{-1} \left(\left(\displaystyle\frac{PN-QM}{PR-QS}\right)(\mu,u,v) + Z(v)\right)  
    \end{array}\right].
\end{equation}
Note that $\mathfrak{F}$ is well defined because, for each $(\mu,u,v) \in \R \times \mathcal{C}_0^1(\overline{\Omega})\times \mathcal{C}_0^1(\overline{\Omega})$, 
$$
P(u,v), Q(u,v), R(u,v), S(u,v), M(u,v),  N(\mu,u,v), Z(v) \in \mathcal{C}(\overline{\Omega}), \quad \forall \lambda \in \R,
$$
and, since the operators
$$(-\Delta)^{-1}, (-\Delta+Z)^{-1}: \mathcal{C}(\overline{\Omega}) \longrightarrow \mathcal{C}_0^1(\overline{\Omega})$$
are well defined, it follows that  $\mathfrak{F}$ is well defined. 
Moreover,  $(u,v) \in W_0^{2,p}(\Omega)\times W_0^{2,p}(\Omega)$ is a non-negative strong solution of (\ref{sistema}) if, and only if,
$$\mathfrak{F}(\mu,u,v) = 0 \quad \mu \in \R.$$
Furthermore,
$$ \mathfrak{F}(\mu,\theta_\lambda,0) = 0  \quad \forall \mu \in \R,$$
consequently $(\mu,\theta_\lambda,0)$ can be regarded as the known curve of the solutions from which we hope the coexistence states will bifurcate.

The next important step to apply the  Theorem 6.4.3 of \cite{Bifbook} is  to calculate   
$$\mathfrak{L}(\mu):=D_{(u,v)} \mathfrak{F}(\mu,\theta_\lambda,0).$$
In the following we will calculate this derivative. By definition of $\mathfrak{F}$ (see (\ref{F})) we have
\begin{multline}\label{Lmu}
\mathfrak{L}(\mu)=D_{(u,v)} \mathfrak{F}(\mu,\theta_\lambda,0) = \cr
\left[\begin{array}{cc}
    I-    (-\Delta)^{-1}\left(\displaystyle\frac{RM-SN}{PR-QS}\right)_u  & -    (-\Delta)^{-1}\left(\displaystyle\frac{RM-SN}{PR-QS}\right)_v\\
     -   (-\Delta + Z)^{-1} \left(\displaystyle\frac{PN-QM}{PR-QS}\right)_u &      I-   (-\Delta + Z)^{-1} \left(\displaystyle\frac{PN-QM}{PR-QS}+Z\right)_v 
    \end{array}\right]    
\end{multline}
where each term is computed in $(\theta_\lambda,0)$ and we have already used that $Z$ does not depend on $u$ and, hence, $Z_u\equiv0$.
Let us compute each term in the above operator. We emphasize again that the functions $P,Q,R,S,M$ and $N$ as well as their derivatives are calculated in $(\theta_\lambda,0)$ and we omit the point by simplicity. Thus,
\begin{eqnarray*}
N&=&N(\mu,\theta_\lambda,0) = Q_u|\nabla \theta_\lambda|^2,\\
N_u\xi &=& N_u(\mu,\theta_\lambda,0)\xi= Q_{uu} \xi |\nabla\theta_\lambda|^2 + 2 Q_u \nabla \theta_\lambda \nabla \xi, \\
N_v\eta&=&N_v(\mu,\theta_\lambda,0)\eta= Q_{vu} |\nabla\theta_\lambda|^2 \eta + (Q_v+R_u)\nabla \theta_\lambda \nabla\eta + \mu b(x)\eta + G(x,\theta_\lambda,0)\theta_\lambda\eta, \\
M&=&M(\theta_\lambda,0) = P_u |\nabla \theta_\lambda|^2 + \lambda a(x)\theta_\lambda + f(x,\theta_\lambda)\theta_\lambda = - P \Delta \theta_\lambda, \\
M_u\xi&=& M_u(\theta_\lambda,0)\xi=P_{uu} |\nabla \theta_\lambda|^2 \xi + 2 P_u \nabla \theta_\lambda \nabla \xi + \lambda a(x)\xi + f(x,\theta_\lambda)\xi+f_u(x,\theta_\lambda)\theta_\lambda\xi,\\
M_v\eta&=&M_v(\theta_\lambda,0)\eta = P_{vu}|\nabla \theta_\lambda|^2\eta+ (P_v+S_u) \nabla \theta_\lambda \nabla\eta + F(x,\theta_\lambda,0) \theta_\lambda \eta.
\end{eqnarray*}
By hypothesis ($H_{PQRS}$), $Q(s,0) = 0$ for all  $s \geq 0$, which implies
\begin{equation}\label{Nu}
    Q(\theta_\lambda,0)=Q_u(\theta_\lambda,0)=Q_{uu}(\theta_\lambda,0)=N(\mu,\theta_\lambda,0)=N_u(\mu,\theta_\lambda,0)=0.
\end{equation}
Thus,
$$\left(\frac{RM - SN}{PR-QS}\right)_u(\mu,\theta_\lambda,0) = \left(\frac{M_uP-MP_u}{P^2}\right)(\theta_\lambda,0).
$$
Substituting the terms $M_u(\theta_\lambda,0), P(\theta_\lambda,0)$ and $P_u(\theta_\lambda,0)$, by a direct calculation we obtain
\begin{multline*}
    \left(\frac{RM - SN}{PR-QS}\right)_u(\mu,\theta_\lambda,0)u = P(\theta_\lambda,0)^{-1}\left\{\mbox{div}(P_u(\theta_\lambda,0) u \nabla \theta_\lambda)+ P_u(\theta_\lambda,0) \nabla \theta_\lambda \nabla u \right. \cr \left. +\lambda a(x)u +[\theta_\lambda f_u(x,\theta_\lambda)+f(x,\theta_\lambda)] u \right\}.
\end{multline*}
Define the operator $T_{1,\lambda}: \mathcal{C}_0^1(\overline{\Omega}) \longrightarrow \mathcal{C}_0^1(\overline{\Omega})$ given by
\begin{multline}\label{T1lambda}
         T_{1,\lambda}u:=u- (-\Delta)^{-1}\left[P(\theta_\lambda,0)^{-1}\left\{\mbox{div}(P_u(\theta_\lambda,0) u \nabla \theta_\lambda)+ P_u(\theta_\lambda,0) \nabla \theta_\lambda \nabla u \right.\right.  \cr \left.\left. +\lambda a(x)u +[\theta_\lambda f_u(x,\theta_\lambda)+f(x,\theta_\lambda)] u \right\}\right], 
\end{multline}
then $T_{1,\lambda}$ is well defined and by above discussion
\begin{equation}\label{DF1}
    I-    (-\Delta)^{-1}\left(\displaystyle\frac{RM-SN}{PR-QS}\right)_u(\theta_\lambda,0) = T_{1,\lambda}.
\end{equation}
Similarly, it follows from (\ref{Nu}) that
\begin{equation*}
    \left(\displaystyle\frac{PN-QM}{PR-QS}\right)_u (\mu,\theta_\lambda,0)=0
\end{equation*}
and, hence, 
\begin{equation}\label{DF2}
(-\Delta + Z)^{-1}\left(\displaystyle\frac{PN-QM}{PR-QS}\right)_u (\mu,\theta_\lambda,0) \equiv 0.
\end{equation}

On the other hand, again by (\ref{Nu}) and using that $Z$ is linear in $v$ and therefore $Z_v=Z$, we deduce that
\begin{eqnarray*}
\left(\displaystyle\frac{PN-QM}{PR-QS} + Z\right)_v(\mu,\theta_\lambda,0) &=& \frac{PN_v - Q_vM}{PR}(\mu,\theta_\lambda,0) + Z(\theta_\lambda)\\
&=& \frac{N_v+ Q_v \Delta \theta_\lambda}{R}(\mu,\theta_\lambda,0) + Z(\theta_\lambda).
\end{eqnarray*}
Substituting $Z(\theta_\lambda)$ and $N_v(\mu,\theta_\lambda,0)$, by a direct calculation we obtain
\begin{multline}\label{bounfra}
    \left(\displaystyle\frac{PN-QM}{PR-QS}+Z\right)_v (\mu,\theta_\lambda,0) = \cr
     \frac{Q_{vu}(\theta_\lambda,0)|\nabla \theta_\lambda|^2 + \mu b(x) + G(x,\theta_\lambda,0)\theta_\lambda + Q_v(\theta_\lambda,0) \Delta \theta_\lambda + k}{R(\theta_\lambda,0)}.
\end{multline} 
Define now the operator $T_\mu: \mathcal{C}_0^1(\overline{\Omega}) \longrightarrow \mathcal{C}_0^1(\overline{\Omega}), ~ \mu \in \R$, given by
\begin{equation}\label{T}
    T_\mu = (-\Delta + Z)^{-1}\left[\frac{Q_{vu}(\theta_\lambda,0)|\nabla \theta_\lambda|^2 + \mu b(x) + G(x,\theta_\lambda,0)\theta_\lambda + Q_v(\theta_\lambda,0) \Delta \theta_\lambda + k}{R(\theta_\lambda,0)}\right],
\end{equation}
then $T_\mu$ is well defined and by above discussion
\begin{equation}\label{DF3}
    (-\Delta + Z)^{-1} \left(\displaystyle\frac{PN-QM}{PR-QS}+Z\right)_v (\mu,\theta_\lambda,0) = T_\mu. 
\end{equation}

Finally, we will denote
\begin{equation}\label{DF4}
    \widehat{T} :=    (-\Delta)^{-1} \left(\displaystyle\frac{RM-SN}{PR-QS}\right)_v (\mu,\theta_\lambda,0).
\end{equation}
By (\ref{DF1}), (\ref{DF2}), (\ref{DF3}) and (\ref{DF4}), we conclude that
$$
\mathfrak{L}(\mu) = \left[\begin{array}{cc}
    T_{1,\lambda} & -\widehat{T} \\
    0 & I-T_{\mu_\lambda} 
\end{array}\right].
$$

\section{Auxiliary Results}\label{sec3}
In this section we will prove some useful properties of the operators $T_{1,\lambda}$ and $T_\mu$ defined in previous section and use them to prove some important results of $\mathfrak{L}(\mu)$.

The first lemma connects the concept of nondegenerate solution of (\ref{stl}) with the invertibility of $T_{1,\lambda}$.

\begin{Lema}\label{lemaT1l}
A solution $(\lambda, \theta_\lambda)$ of (\ref{stl}) is nondegenerate if, and only if,  $T_{1,\lambda}$ is invertible.
\end{Lema}
\begin{proof}
Recall that by Definition~\ref{degenerada},  $(\lambda,\theta_\lambda)$ is nondegenerate solution of (\ref{stl}) is zero is the unique strong solution of (\ref{ndegen}). On the other hand, it is equivalent to
\begin{multline*}
         T_{1,\lambda}u=u- (-\Delta)^{-1}\left[P(\theta_\lambda,0)^{-1}\left\{\mbox{div}(P_u(\theta_\lambda,0) u \nabla \theta_\lambda)+ P_u(\theta_\lambda,0) \nabla \theta_\lambda \nabla u \right.\right.  \cr \left.\left. +\lambda a(x)u +[\theta_\lambda f_u(x,\theta_\lambda)+f(x,\theta_\lambda)] u \right\}\right]=0, \quad (u \in \mathcal{C}_0^1(\overline{\Omega})).
\end{multline*}
That is,  $(\lambda,\theta_\lambda)$ is a nondegenerate solution of (\ref{stl}) if, and only if, $N[T_{1,\lambda}]=\{0\}$. Since $T_{1,\lambda}$ is a linear operator which is a compact perturbation of the identity map, it follows that it is equivalent to say that $T_{1,\lambda}$ is invertible.
\end{proof}

Now, we will prove some properties of the operator $T_\mu$.

\begin{Lema}\label{propT}
For $k>0$ large enough, the operator $T_{\mu_\lambda}$ (see  (\ref{T})) where  
$$\mu=\mu_\lambda=\sigma_1\left[-\mbox{div}\left(Q_v(\theta_\lambda,0)\nabla \theta_\lambda+R(\theta_\lambda,0) \nabla\right) - G(x,\theta_\lambda,0)\theta_\lambda;b \right],$$
is strongly positive and satisfies
$$r(T_{\mu_\lambda})=1,$$
where $r(T_{\mu_\lambda})$ denotes the spectral radius of $T_{\mu_\lambda}$.
\end{Lema}
\begin{proof}
Choose $k>0$ sufficiently large such that
$$\frac{Q_{vu}(\theta_\lambda,0)|\nabla \theta_\lambda|^2 + \mu_\lambda b(x)+ G(x,\theta_\lambda,0)\theta_\lambda + Q_v(\theta_\lambda,0) \Delta \theta_\lambda + k}{R(\theta_\lambda,0)}\geq 0.$$
Then, for $v> 0$, $y = T_{\mu_\lambda}v$  satisfies
\begin{eqnarray*}
-\Delta y +Z(y) &=& \frac{Q_{vu}(\theta_\lambda,0)|\nabla \theta_\lambda|^2v + \mu_\lambda b(x)v + G(x,\theta_\lambda,0)\theta_\lambda v + Q_v(\theta_\lambda,0) v\Delta \theta_\lambda + kv}{R(\theta_\lambda,0)}\\
&\geq& 0.
\end{eqnarray*}
Since positive constants are strict supersolution of  $-\Delta + Z$ in $\Omega$ under homogeneous Dirichlet boundary conditions, then it satisfies the Strong Maximum Principle. Consequently, $y \in \mbox{int}({\cal P})$, showing that $T_{\mu_\lambda}$ is strongly positive.

To establish that $r(T_{\mu_\lambda})=1$, we argue as follows. Let $\varphi_2$ be the positive eigenfunction associated to 
$$\mu_\lambda=\sigma_1\left[-\mbox{div}\left(Q_v(\theta_\lambda,0)\nabla \theta_\lambda+R(\theta_\lambda,0) \nabla\right) - G(x,\theta_\lambda,0)\theta_\lambda;b \right]$$
with $\|\varphi_2\|_0=1$.
We claim that 
$$T_{\mu_\lambda}\varphi_2 = \varphi_2.$$
Indeed, this equality is equivalent to
\begin{multline*}
    -\Delta \varphi_2 +Z(\varphi_2) =\cr \frac{Q_{vu}(\theta_\lambda,0)|\nabla \theta_\lambda|^2\varphi_2 + \mu_\lambda b(x) \varphi_2 + G(x,\theta_\lambda,0)\theta_\lambda \varphi_2 + Q_v(\theta_\lambda,0) \varphi_2\Delta \theta_\lambda + k\varphi_2}{R(\theta_\lambda,0)}.
\end{multline*}
Substituting $Z(\varphi_2)$, by a direct calculation, we find that
\begin{equation}
    -\mbox{div}\left(Q_v(\theta_\lambda,0)\varphi_2\nabla \theta_\lambda+R(\theta_\lambda,0) \nabla\varphi_2\right) - G(x,\theta_\lambda,0)\theta_\lambda\varphi_2 = \mu_\lambda b(x)\varphi_2
\end{equation}
which holds true by definition of $\varphi_2$. Thus, $1$ is a eigenvalue of $T_{\mu_\lambda}$ with an associated positive eigenfunction. Since $T_{\mu_\lambda}$ is strong positive and $\mathcal{C}_0^1(\overline{\Omega})$ has positive cone with nonempty interior, by Theorem 12.3 of \cite{KreinR}, we conclude that $r(T_{\mu_\lambda})=1$.
\end{proof}

The next corollary  will be very useful later.

\begin{Cor}\label{Tsol}
Let $y \in \mathcal{C}_0^1(\overline{\Omega})\setminus\{0\}$ with $y\geq0$. Then the equation
$$v-T_{\mu_\lambda}v = y$$
has no solution $v \in \mathcal{C}_0^1(\overline{\Omega})$.
\end{Cor}
\begin{proof}
Since $T_{\mu_\lambda}$ is strongly positive with $r(T_{\mu_\lambda})=1$ and $\mathcal{C}_0^1(\overline{\Omega})$ has positive cone with nonempty interior, the result follows from  Corollary 12.4 of \cite{KreinR}.
\end{proof}

\begin{remark}
From now on we shall assume that $(\lambda,\theta_\lambda)$ is a  nondegenerate strong solution of (\ref{stl}) with $\theta_\lambda \in \mbox{int}({\cal P})$ and $k>0$ fixed such that the statements of Lemma~\ref{propT} holds true.
\end{remark}

The next step is to find a 1-transversal eigenvalues of $\mathfrak{L}(\mu)$ (in the sense of \cite{Bifbook}).
As the natural candidate is $\mu=\mu_\lambda$, let us determine the range and the kernel of $\mathfrak{L}(\mu_\lambda)$. 

\begin{Pro}\label{conda}
\begin{enumerate}
    \item[(a)] $$N[\mathfrak{L}(\mu_\lambda)] = \mbox{span} \langle(T_{1,\lambda}^{-1}(\widehat{T}\varphi_2),\varphi_2) \rangle$$
where $\varphi_2$ stands for the positive eigenfunction associated to $\mu_\lambda$ with $\|\varphi_2\|_0=1$.
    \item[(b)] $$R[\mathfrak{L}(\mu_\lambda)] = \mathcal{C}_0^1(\overline{\Omega}) \times R[I - T_{\mu_\lambda}]$$
where $T_{\mu_\lambda}$ denotes the operator defined by (\ref{T}).
\end{enumerate}
\end{Pro}
\begin{proof}
To prove (a), observe that $(\xi, \eta) \in N[\mathfrak{L}(\mu_\lambda)]$ is equivalent to
\begin{equation*}
    \mathfrak{L}(\mu_\lambda) (\xi,\eta)^t = (0,0)^t 
\end{equation*}
\begin{equation}\label{AAA}
    \left[\begin{array}{cc}
    T_{1,\lambda} & -\widehat{T} \\
    0 & I-T_{\mu_\lambda} 
\end{array}\right]  
\left[\begin{array}{c}
    \xi \\
    \eta  
\end{array}\right] = \left[\begin{array}{c}
    0 \\
    0  
\end{array}\right].
\end{equation}
It follows from  second equation of (\ref{AAA}) that
$$T_{\mu_\lambda}\eta = \eta.$$
By Lemma~\ref{propT} we have $r(T_{\mu_\lambda})=1$. Therefore, the above equality is an eigenvalue problem whose solutions are  $c\varphi_2$,  $c \in \R$. Consequently,
$$ \eta \in \mbox{span}\langle\varphi_2\rangle.$$
On the other hand, the first equation of (\ref{AAA}) with $\eta = \varphi_2$ yields
$$T_{1,\lambda} \xi = \widehat{T}\varphi_2.$$
Once that $T_{1,\lambda}$ is invertible (according with Lemma~\ref{lemaT1l}), we obtain
$$\xi = T_{1,\lambda}^{-1}(\widehat{T}\varphi_2).$$
Thus,
$$N[\mathfrak{L}(\mu_\lambda)] = \mbox{span} \langle(T_{1,\lambda}^{-1}(\widehat{T}\varphi_2),\varphi_2)\rangle.$$

Now let us prove (b). $(\xi, \eta) \in R[\mathfrak{L}(\mu_\lambda)]$ is equivalent to say that there exists $(u,v) \in \mathcal{C}_0^1(\overline{\Omega}) \times \mathcal{C}_0^1(\overline{\Omega})$ such that
$$
\mathfrak{L}(\mu_\lambda)(u,v)^t = (\xi, \eta)^t
$$
or equivalently, 
\begin{equation}\label{BBB}
    \left[\begin{array}{cc}
    T_{1,\lambda} & -\widehat{T} \\
    0 & I-T_{\mu_\lambda} 
\end{array}\right]  
\left[\begin{array}{c}
    u \\
    v  
\end{array}\right] = \left[\begin{array}{c}
    \xi \\
    \eta 
\end{array}\right].
\end{equation}
It follows from second equation of (\ref{BBB}) that
$$\eta \in R[I-T_{\mu_\lambda}].$$
While the first one gives us
\begin{equation}\label{1125}
T_{1,\lambda}u = \widehat{T}v + \xi .
\end{equation}
Recall that $\widehat{T}$ is operator $\mathcal{C}_0^1(\overline{\Omega})$ into $\mathcal{C}_0^1(\overline{\Omega})$. Then for each $v, \xi \in \mathcal{C}_0^1(\overline{\Omega})$, 
$$ \widehat{T}v + \xi \in \mathcal{C}_0^1(\overline{\Omega}).$$
Since $T_{1,\lambda}:\mathcal{C}_0^1(\overline{\Omega}) \rightarrow \mathcal{C}_0^1(\overline{\Omega})$ is invertible (see Lemma~\ref{T1lambda}), we conclude that there exists $u \in \mathcal{C}_0^1(\overline{\Omega})$ satisfying (\ref{1125}). Therefore
$$R[\mathfrak{L}(\mu_\lambda)] = \mathcal{C}_0^1(\overline{\Omega}) \times R[I - T_{\mu_\lambda}].
$$
\end{proof}

To verify the transversality condition of \cite{RabC}, it remains to determine $$\mathfrak{L}_1:=D_\mu\mathfrak{L}(\mu_\lambda).$$

\begin{Pro}
\label{conda2}
$$\mathfrak{L}_1:=D_\mu\mathfrak{L}(\mu_\lambda) =  \left[\begin{array}{cc}
     0& \displaystyle -(-\Delta)^{-1}\left(\frac{S(\theta_\lambda,0)b(x)}{(PR)(\theta_\lambda,0)}\right)\\
     0& \displaystyle -(-\Delta + Z)^{-1}\left(\frac{b(x)}{R(\theta_\lambda,0)}\right)
\end{array}\right]$$
\end{Pro}
\begin{proof}
Taking into account  that only  $N_v$ depends on $\mu$, by a direct calculation we have
$$ N_{\mu v} \equiv b(x)$$
and the derivatives with respect to $\mu$ of all other term that appear in (\ref{Lmu}) are zero. Then, differentiating  $\mathfrak{L}(\mu)$ with respect to $\mu$ at $\mu_\lambda$   yields
$$\mathfrak{L}_1=D_\mu\mathfrak{L}(\mu_\lambda)  = \left[\begin{array}{cc}
     0& \displaystyle -(-\Delta)^{-1}\left(\frac{S(\theta_\lambda,0)b(x)}{(PR)(\theta_\lambda,0)}\right)\\
     0& \displaystyle -(-\Delta + Z)^{-1}\left(\frac{b(x)}{R(\theta_\lambda,0)}\right)
\end{array}\right].$$
\end{proof}

Finally, we conclude this section by providing that $\mu_\lambda$ is a 1-transversal eigenvalue of  the family $\mathfrak{L}(\mu)$ and an useful characterization  of the complement $N[\mathfrak{L}(\mu_\lambda)]$ on $\mathcal{C}_0^1(\overline{\Omega}) \times \mathcal{C}_0^1(\overline{\Omega})$.

\begin{Pro}\label{condd}
\begin{enumerate}
\item[(a)]
\begin{equation}\label{1transversal2}
\mathfrak{L}_1(N[\mathfrak{L}(\mu_\lambda)]) \oplus R[\mathfrak{L}(\mu_\lambda)] = \mathcal{C}_0^1(\overline{\Omega}) \times \mathcal{C}_0^1(\overline{\Omega}).
\end{equation}
\item[(b)]\begin{equation*}
    N[\mathfrak{L}(\mu_\lambda)] \oplus R[\mathfrak{L}(\mu_\lambda)]= \mathcal{C}_0^1(\overline{\Omega}) \times \mathcal{C}_0^1(\overline{\Omega}).
\end{equation*}
\end{enumerate}
\end{Pro}
\begin{proof}
To prove paragraph (a), recall that by Proposition \ref{conda}
$$N[\mathfrak{L}(\mu_\lambda)] = \mbox{span} \langle(T_{1,\lambda}^{-1}(\widehat{T}\varphi_2),\varphi_2)\rangle $$
and by Proposition \ref{conda2}
$$\mathfrak{L}_1 =  \left[\begin{array}{cc}
     0& \displaystyle -(-\Delta)^{-1}\left(\frac{S(\theta_\lambda,0)b(x)}{(PR)(\theta_\lambda,0)}\right)\\
     0& \displaystyle -(-\Delta + Z)^{-1}\left(\frac{b(x)}{R(\theta_\lambda,0)}\right)
\end{array}\right].$$
Thus,  $\mathfrak{L}_1(N[\mathfrak{L}(\mu_\lambda)])$ is given by
$$
\left[\begin{array}{cc}
     0& \displaystyle -(-\Delta)^{-1}\left(\frac{S(\theta_\lambda,0)b(x)}{(PR)(\theta_\lambda,0)}\right)\\
     0& \displaystyle -(-\Delta + Z)^{-1}\left(\frac{b(x)}{R(\theta_\lambda,0)}\right)
\end{array}\right] \left[\begin{array}{c}
     T_{1,\lambda}^{-1}\widehat{T}(c\varphi_2)\\
     c\varphi_2
\end{array}\right] = 
\left[\begin{array}{c}
     \displaystyle -(-\Delta)^{-1}\left(\frac{cS(\theta_\lambda,0)b(x) \varphi_2}{(PR)(\theta_\lambda,0)}\right)\\
     \displaystyle -(-\Delta + Z)^{-1}\left(\frac{cb(x)\varphi_2}{R(\theta_\lambda,0)}\right)
\end{array}\right],
$$
where $c \in \R$. That is,
$$\mathfrak{L}_1(N[\mathfrak{L}(\mu_\lambda)]) = \mbox{span} \left\langle
\left[\begin{array}{c}
     \displaystyle (-\Delta)^{-1}\left(\frac{S(\theta_\lambda,0)b(x) \varphi_2}{(PR)(\theta_\lambda,0)}\right)\\
     \displaystyle (-\Delta + Z)^{-1}\left(\frac{b(x)\varphi_2}{R(\theta_\lambda,0)}\right)
\end{array}\right] \right\rangle.$$
Since $\mathfrak{L}(\mu)$ is a Fredholm operator of index zero (because it is a compact perturbation of identity map), we have
$$\mbox{codim}R[\mathfrak{L}(\mu_\lambda)] = \mbox{dim} N[\mathfrak{L}(\mu_\lambda)] = 1$$
and, hence, the complement of $R[\mathfrak{L}(\mu_\lambda)]$ in $\mathcal{C}_0^1(\overline{\Omega})\times \mathcal{C}_0^1(\overline{\Omega})$ is one dimensional. Then, to establish (\ref{1transversal2}) it is sufficient to show
\begin{equation}\label{564}
\left[\begin{array}{c}
     \displaystyle (-\Delta)^{-1}\left(\frac{S(\theta_\lambda,0) \varphi_2}{(PR)(\theta_\lambda,0)}\right)\\
     \displaystyle (-\Delta + Z)^{-1}\left(\frac{b(x)\varphi_2}{R(\theta_\lambda,0)}\right)
\end{array}\right] \not\in R[\mathfrak{L}(\mu_\lambda)]= \mathcal{C}_0^1(\overline{\Omega}) \times R[I-T_{\mu_\lambda}].
\end{equation}
To prove (\ref{564}) we proceed by contradiction. If (\ref{564}) fails, then 
$$\displaystyle (-\Delta + Z)^{-1}\left(\frac{b(x)\varphi_2}{R(\theta_\lambda,0)}\right) \in R[I-T_{\mu_\lambda}]$$
and, hence, there exists $v \in \mathcal{C}_0^1(\overline{\Omega})$ such that
\begin{equation}\label{0147}
v - T_{\mu_\lambda}v =  \displaystyle (-\Delta + Z)^{-1}\left(\frac{b(x)\varphi_2}{R(\theta_\lambda,0)}\right).
\end{equation}
On the other hand, since $(-\Delta + Z)^{-1}$ satisfies the Strong Maximum Principle and 
$$b(x)\varphi_2/R(\theta_\lambda,0) > 0,$$ 
we find that
$$(-\Delta + Z)^{-1}\left(\frac{b(x)\varphi_2}{R(\theta_\lambda,0)}\right) \geq0.$$
By Corollary~\ref{Tsol}, (\ref{564})  has no solution  $v \in \mathcal{C}_0^1(\overline{\Omega})$, which is a contradiction. This completes the proof of (a). The proof of (b) is rather similar and so we omit it.

\end{proof}

\section{Proof of Theorem~\ref{Thsistema}}\label{sec4}

In the previous sections we have collected the necessary results to apply Theorem 6.4.3 of \cite{Bifbook}. We are now able to provide the

\begin{demThs}
We already know that  $\mathfrak{L}(\mu)$  is a compact perturbation of identity map for each $\mu \in \R$ and it is analytic in $\mu$. Moreover, by Proposition~\ref{condd} (a), $\mu_\lambda$ is a 1-transversal eigenvalue of $\mathfrak{L}(\mu)$ and  by Proposition~\ref{conda},
$$N[\mathfrak{L}(\mu_\lambda)] = \mbox{span} \langle(T_{1,\lambda}^{-1}(\widehat{T}\varphi_2),\varphi_2)\rangle,$$
in particular, $\mbox{dim}N[\mathfrak{L}(\mu_\lambda)]=1$. Then, the algebraic multiplicity of $\mathfrak{L}(\mu)$ satisfies
$$\chi[\mathfrak{L}(\mu);\mu_\lambda]=1.$$
Thus, owing to Theorem 5.6.2 of \cite{Bifbook}, the index Leray-Schauder of $\mathfrak{L(\mu)}$ as a compact perturbation of the identity map changes sign as $\lambda$ crosses $\mu_\lambda$. Consequently, we can apply the unilateral bifurcation result of  \cite{Bifbook} (see Theorem 6.4.3).

Let $\mathcal{S}$ denote the set of points $(\mu,u,v) \in \mathbb{R} \times \mathcal{C}_0^1(\overline{\Omega}) \times \mathcal{C}_0^1(\overline{\Omega})$ such that
$$ \mathfrak{F}(\mu,u,v) =0$$
and either $(u,v) \neq (\theta_\lambda,0)$ or $(u,v) = (\theta_\lambda,0)$ and $\mathfrak{L}(\mu)$ is not invertible. Let $\mathfrak{C}^+$ and $\mathfrak{C}^-$ denote the components of $\mathcal{S}$ whose  existence are guaranteed by Proposition 6.4.2 of \cite{Bifbook}. Basically, $\mathfrak{C}^+$ (resp. $\mathfrak{C}^-)$ is a subcontinuum that near to $(\mu_\lambda,\theta_\lambda,0)$ belongs to the positive cone (resp. negative cone).
According to Theorem 6.4.3 of \cite{Bifbook}, one of the following non-excluding options occurs. Either
\begin{enumerate}[{A}1.]
    \item $\mathfrak{C}^+$ is unbounded in $\mathbb{R} \times \mathcal{C}_0^1(\overline{\Omega})\times \mathcal{C}_0^1(\overline{\Omega})$.
    \item There exists another eigenvalue of $\mathfrak{L}(\mu)$, e.g., $\widehat{\mu} \neq \mu_\lambda$, such that 
    $$(\widehat{\mu},\theta_\lambda,0) \in \mathfrak{C}^+.$$
    \item There exists $\mu \in\mathbb{R}$ and $y \in Y\setminus\{0\}$ such that $(\mu,y) \in \mathfrak{C}^+$, where $Y$ is the complement of $N[\mathfrak{L}(\mu_\lambda)$ in $\mathcal{C}_0^1(\overline{\Omega})\times \mathcal{C}_0^1(\overline{\Omega})$.
\end{enumerate}

By Proposition~\ref{condd} (b), we can choose
$$Y:=R[\mathfrak{L}(\mu_\lambda)]= \mathcal{C}_0^1(\overline{\Omega}) \times R[I-T_{\mu_\lambda}].$$
Moreover, by Lemma 6.4.1 of \cite{Bifbook}, the solutions of $\mathfrak{C}^+$ in a neighborhood of $(\mu,\theta_\lambda,0)$ are coexistence states, since $\theta_\lambda$ and $\varphi_2$ are functions belonging to $\mbox{int}({\cal P})$. Let $\mathfrak{C}$ denote the subcontinuum of $\mathfrak{C}^+$ satisfying
$$
\mathfrak{C} \subset \mathbb{R} \times \mbox{int}({\cal P})\times \mbox{int}({\cal P}).
$$

If $\mathfrak{C}$ is unbounded, then alternative 1 of the statement of Theorem~\ref{Thsistema} is satisfied, completing the proof in this case.

During the rest of the proof we assume that $\mathfrak{C}$ is bounded in $\R \times \mathcal{C}_0^1(\overline{\Omega})\times \mathcal{C}_0^1(\overline{\Omega})$. 

\underline{Suppose $\mathfrak{C}=\mathfrak{C}^+$.}

Then $\mathfrak{C}^+$ is bounded and, hence, alternative A1 is not satisfied. Suppose that alternative A3 happens. Then
$$
(\mu,y) \in \mathfrak{C}^+ = \mathfrak{C} \subset \R \times \mbox{int}({\cal P})\times \mbox{int}({\cal P})
$$
for some $y=(y_1,y_2) \in Y = \mathcal{C}_0^1(\overline{\Omega}) \times R[I-T_{\mu_\lambda}]$. In particular,
$$y_2 \in R[I-T_{\mu_\lambda}] \cap \mbox{int}({\cal P})$$
which implies that there exists $v \in \mathcal{C}_0^1(\overline{\Omega})$ such that
$$v-T_{\mu_\lambda}v=y \geq 0,$$
a contradiction with Corollary~\ref{Tsol}. Consequently, $\mathfrak{C}^+$ satisfies alternative A2. In particular, there exist two bifurcation points of coexistence states of (\ref{sistema}): $(\mu_\lambda, \theta_\lambda,0)$ and $(\widehat{\mu},\theta_\lambda,0)$. Hence, there exists a sequence $(\mu_n, u_n, v_n)$ of coexistence states of (\ref{sistema}) such that
$$(\mu_n, u_n, v_n) \rightarrow  (\widehat{\mu},\theta_\lambda,0) \quad \mbox{in}~ \R \times \mathcal{C}_0^1(\overline{\Omega})\times \mathcal{C}_0^1(\overline{\Omega}).$$
Now, consider 
$$
 \widehat{v}_n:= \frac{v_n}{\|v_n\|_0}, \quad n \geq 1.
 $$
Since $\mathfrak{F}(\mu_n, u_n, v_n)=0$ and in view of (\ref{F}), it follows that $(\mu_n, u_n, \widehat{v}_n)$ satisfies
\begin{equation}\label{vnn}
\widehat{v}_n =  (-\Delta + Z)^{-1} \left[\frac{1}{\|v_n\|_0}\left(\displaystyle\frac{PN-QM}{PR-QS}\right)(\mu_n,u_n,v_n) + Z(\widehat{v}_n)\right].
\end{equation}
Once that $\|\widehat{v}_n\|_0=1$ and $(\mu_n,u_n,v_n)$ is bounded in $\R\times\mathcal{C}_0^1(\overline{\Omega}) \times \mathcal{C}_0^1(\overline{\Omega})$ (because it converges), we obtain that
$$
\frac{1}{\|v_n\|_0}\left(\displaystyle\frac{PN-QM}{PR-QS}\right)(\mu_n,u_n,v_n) + Z(\widehat{v}_n) 
$$
is bounded in $\mathcal{C}(\overline{\Omega})$ (recall (\ref{bounfra})). From compactness of the operator  $(-\Delta + Z)^{-1}: \mathcal{C}(\overline{\Omega}) \rightarrow \mathcal{C}_0^1(\overline{\Omega})$, up to a subsequence if  necessary,
$$ \widehat{v}_n \rightarrow  w \quad \mbox{in} ~\mathcal{C}_0^1(\overline{\Omega}),$$
with $\|w\|_0=1$ and $w>0$. In order to take the limit $n \rightarrow \infty$ in (\ref{vnn}) we proceed as follows. Since  $N(\mu,u,0)=0$ for all $\mu \in \R$ and $u \in \mathcal{C}_0^1(\overline{\Omega}),~u>0$, we have
\begin{eqnarray*}
\lim_{n \rightarrow \infty}\frac{N(\mu_n,u_n,v_n)}{\|v_n\|_0} &=& \lim_{n \rightarrow \infty}\frac{N(\mu_n,u_n,\widehat{v}_n\|v_n\|_0)}{\widehat{v}_n\|v_n\|_0} \widehat{v}_n \\
&=& N_v(\widehat{\mu},\theta_\lambda,0)w.
\end{eqnarray*}
Analogously,
$$\lim_{n \rightarrow \infty}\frac{Q(u_n,v_n)}{\|v_n\|_0} = Q_v(\theta_\lambda,0)w.$$
Therefore,
\begin{multline*}
\frac{PN-QM}{\|v_n\|_0}(\mu_n,u_n,v_n) = P(u_n,v_n)\frac{N(\mu_n,u_n,v_n)}{\|v_n\|_0}-\frac{Q(u_n,v_n)}{\|v_n\|_0}M(u_n,v_n) \longrightarrow \cr P(\theta_\lambda,0)N_v(\widehat{\mu},\theta_\lambda,0)w- Q_v(\theta_\lambda,0)w M(\theta_\lambda,0) \quad \mbox{as}~n \rightarrow \infty,
\end{multline*}
Moreover,
$$(PR-QS)(u_n,v_n) \longrightarrow (PR-QS)(\theta_\lambda,0) = (PR)(\theta_\lambda,0) \quad\mbox{as}~n \rightarrow \infty$$
and
$$Z(\widehat{v}_n) \longrightarrow Z(v) \quad\mbox{as}~n \rightarrow \infty.$$
Thus, letting $n \rightarrow \infty$ in (\ref{vnn}) yields
$$ w =  (-\Delta + Z)^{-1} \left[\left(\frac{PN_v w-Q_vwM}{PR}\right)(\widehat{\mu},\theta_\lambda,0) + Z(w)\right] .$$
In particular, by elliptic regularity, $w \in W_0^{2,p}(\Omega)$, for all $p>1$. Substituting $M(\theta_\lambda,0)= -P\Delta\theta_\lambda$ and $Z(w)$, the above equality is equivalent to
\begin{equation}\label{ppppp}
    \left\{ \begin{array}{ll}
         -\mbox{div}(Q_v(\theta_\lambda,0)w\nabla \theta_\lambda+R(\theta_\lambda,0) \nabla w) = \widehat{\mu} b(x) w + G(x,\theta_\lambda,0)\theta_\lambda w & \mbox{in}~\Omega,\\
         w=0 &\mbox{on}~ \partial \Omega.
    \end{array} \right.
\end{equation}
Since $w>0$ and by uniqueness of the principal eigenvalue, (\ref{ppppp}) implies that $\widehat{\mu}=\mu_\lambda$, which is a contradiction, showing that $\mathfrak{C}=\mathfrak{C}^+$ cannot occur. Consequently:\\
$\underline{\mathfrak{C} \subset \mathfrak{C}^+ \quad \mbox{and} \quad \mathfrak{C} \neq \mathfrak{C}^+.}$\\
We now establish that  $\mathfrak{C}$ satisfies either alternative 2 or 3 or 4 of Theorem~\ref{Thsistema}. Since $\mathfrak{C}$ is a proper subset of  $\mathfrak{C}^+$, there exists $(\mu^*,u^*,v^*) \in \R \times \partial ({\cal P} \times {\cal P})$
and a sequence 
$$(\mu_n,u_n,v_n) \in \R \times\mbox{int}({\cal P}) \times \mbox{int}({\cal P}) $$
such that
$$(\mu_n,u_n,v_n) \rightarrow (\mu^*,u^*,v^*) \quad \mbox{in}~ \R \times \mathcal{C}_0^1(\overline{\Omega})\times \mathcal{C}_0^1(\overline{\Omega}).$$
By continuity  $\mathfrak{F}$,
$$0=\mathfrak{F}(\mu_n,u_n,v_n) \rightarrow \mathfrak{F}(\mu^*,u^*,v^*),$$
that is, $(\mu^*,u^*,v^*)$ is a non-negative solution of (\ref{sistema}). Moreover, since
$$
(u^*,v^*) \in \partial ({\cal P} \times {\cal P}),
$$
by the Strong Maximum Principle, $u^*= 0$ or $v^*= 0$. Let us consider the three possible cases for this statement:

\underline{Case 1: $u^*= 0$ and $v^*> 0$}

Define
$$\widehat{u}_n :=\frac{u_n}{\|u_n\|_0}, \quad n\geq1.$$
Then, in view of (\ref{F}) and $\mathfrak{F}(\mu_n,u_n,v_n)=0$, $(\mu_n,\widehat{u}_n,v_n)$ satisfies
\begin{equation}\label{unn}
    \left\{\begin{array}{l}
          \widehat{u}_n=    (-\Delta)^{-1}\left[\displaystyle\frac{1}{\|u_n\|_0}\left(\frac{RM-SN}{PR-QS}\right)(\mu_n,u_n,v_n)\right],\vspace{0.5cm} \\
          v_n =  (-\Delta + Z)^{-1} \left[\left(\displaystyle\frac{PN-QM}{PR-QS}\right)(\mu_n,u_n,v_n) + Z(v_n)\right].
    \end{array}\right.
\end{equation}
Since $\|\widehat{u}_n\|_0=1$ and $(\mu_n,u_n,v_n)$ is bounded in $\R\times\mathcal{C}_0^1(\overline{\Omega}) \times \mathcal{C}_0^1(\overline{\Omega})$ (because it converges), we conclude that 
$$\frac{1}{\|u_n\|_0}\left(\displaystyle\frac{RM-SN}{PR-QS}\right)(\mu_n,u_n,v_n) $$
is bounded in $\mathcal{C}(\overline{\Omega})$. Owing to compactness of  $(-\Delta)^{-1}:\mathcal{C}(\overline{\Omega}) \rightarrow \mathcal{C}_0^1(\overline{\Omega})$, it becomes apparent that, up to a subsequence if  necessary,
$$\widehat{u}_n \rightarrow z \quad \mbox{in}~\mathcal{C}_0^1(\overline{\Omega}), $$
with $\|z\|_0=1$ and $z>0$. Letting $n \rightarrow \infty$ in (\ref{unn}) yields
\begin{equation}\label{zv*}
    \left\{\begin{array}{l}
          z=    (-\Delta)^{-1}\left[\left(\displaystyle\frac{RM_uz-S_uzN}{PR}\right)(\mu^*,0,v^*)\right],\vspace{0.5cm} \\
          v^* =  (-\Delta + Z)^{-1} \left[\left(\displaystyle\frac{PN-QM}{PR}\right)(\mu^*,0,v^*) + Z(v^*)\right].
    \end{array}\right.
\end{equation}
Thus, $z,v^* \in \mathcal{C}_0^{1}(\overline{\Omega})$. Moreover, it follows from second equation of (\ref{zv*}) that
\begin{equation}\label{951}
    \left\{ \begin{array}{ll}
         -\mbox{div}(R(0,v^*) \nabla v^*) = \mu^* b(x) v^* + g(x,v^*)v^* &\mbox{in}~\Omega,\\
         v^*=0 & \mbox{on}~\partial \Omega.
    \end{array} \right.
\end{equation}
That is, $(\mu^*,v^*)$ is a solution of (\ref{stm}). Using that $(\mu^*,v^*)$ satisfies (\ref{951}), the first equation of (\ref{zv*}) is equivalent to
\begin{equation*}
    \left\{ \begin{array}{ll}
         -\mbox{div}(P(0,v^*) \nabla z + S_u(0,v^*)z \nabla v^*) =\lambda a(x)z +  F(x,0,v^*)zv^* & \mbox{in}~\Omega,\\
         z=0 & \mbox{on}~\partial \Omega.
    \end{array} \right.
\end{equation*}
Since $z>0$, we conclude that
$$\lambda = \sigma_1[-\mbox{div}(P(0,v^*) \nabla  + S_u(0,v^*) \nabla v^*) - F(x,0,v^*)v^*;a].$$
Therefore, in this case, alternative 2 is satisfied with
$$\theta_{\mu^*}:=v^*.$$

\underline{Case 2: $u^*>0$ and $v^*=0$}

Define
$$\widehat{v}_n:= \frac{v_n}{\|v_n\|_0},\quad n\geq1.$$
Then, in view of (\ref{F}) and $\mathfrak{F}(\mu_n,u_n,v_n)=0$, $(\mu_n,u_n,\widehat{v}_n)$  satisfies
\begin{equation}\label{8858}
    \left\{\begin{array}{l}
          u_n=    (-\Delta)^{-1}\left[\left(\displaystyle\frac{RM-SN}{PR-QS}\right)(\mu_n,u_n,v_n)\right],\vspace{0.5cm} \\
          \widehat{v}_n =  (-\Delta + Z)^{-1} \left[\displaystyle\frac{1}{\|v_n\|_0}\left(\frac{PN-QM}{PR-QS}\right)(\mu_n,u_n,v_n) + Z(\widehat{v}_n)\right].
    \end{array}\right.
\end{equation}
Since $\|\widehat{v}_n\|_0=1$ and $(\mu_n,u_n,v_n)$ is bounded in $\R \times \mathcal{C}_0^1(\overline{\Omega})\times \mathcal{C}_0^1(\overline{\Omega})$ (because it converges), we conclude that
$$ \displaystyle\frac{1}{\|v_n\|_0}\left(\frac{PN-QM}{PR-QS}\right)(\mu_n,u_n,v_n) + Z(\widehat{v}_n)$$
is bounded in $\mathcal{C}(\overline{\Omega})$. Owing to compactness of $(-\Delta + Z)^{-1}: \mathcal{C}(\overline{\Omega}) \rightarrow \mathcal{C}_0^1(\overline{\Omega})$ it becomes apparent that, up to a subsequence  if  necessary,
$$\widehat{v}_n \rightarrow w \quad \mbox{in}~\mathcal{C}_0^1(\overline{\Omega}),$$
with $\|w\|_0 =1$ and $w>0$. Letting $n \rightarrow \infty$ in (\ref{8858}) yields
\begin{equation}\label{vnn0}
    \left\{\begin{array}{l}
          u^*=    (-\Delta)^{-1}\left[\left(\displaystyle\frac{RM-SN}{PR}\right)(\mu^*,u^*,0)\right],\vspace{0.5cm} \\
          w =  (-\Delta + Z)^{-1} \left[\displaystyle\left(\frac{PN_vw-Q_vwM}{PR}\right)(\mu^*,u^*,0) + Z(w)\right].
    \end{array}\right.
\end{equation}
In particular, by elliptic regularity, $u^*,w \in W_0^{2,p}(\Omega)$, for all $p>1$. Moreover, it follows from the first equation of (\ref{vnn0}) that
\begin{equation}\label{952}
    \left\{ \begin{array}{ll}
         -\mbox{div}(P(u^*,0) \nabla u^* ) =\lambda a(x)u^* +  f(x,u^*)u^* & \mbox{in}~\Omega,\\
         u^*=0 & \mbox{on}~\partial \Omega.
    \end{array} \right.
\end{equation}
That is, $(\lambda,u^*)$ is a positive solution of (\ref{stl}). Using that $(\lambda,u^*)$ satisfies (\ref{952}), the second equation of (\ref{vnn0}) is equivalent to
\begin{equation*}
    \left\{ \begin{array}{ll}
         -\mbox{div}(Q_v(u^*,0)w\nabla u^*+R(u^*,0) \nabla w) = \mu^* b(x) w + G(x,u^*,0)u^*w &\mbox{in}~\Omega,\\
         w=0 & \mbox{on}~\partial \Omega.
    \end{array} \right.
\end{equation*}
Since $w>0$, we conclude that
\begin{equation}\label{32147}
\mu^* = \sigma_1[-\mbox{div}(Q_v(u^*,0)\nabla u^*+R(u^*,0) \nabla)- G(x,u^*,0)u^*;b].
\end{equation}
On the other hand, by construction, there exists $\delta>0$ such that
$$
    (\mu^*,u^*,0) \in \mathfrak{C}^+ \setminus B_\delta(\mu_\lambda, \theta_\lambda,0).
$$
Indeed, for $\delta>0$ small enough we have $\mathfrak{C} = \mathfrak{C}^+$ in $ B_\delta(\mu_\lambda, \theta_\lambda,0)$. In particular,
\begin{equation}\label{654789}
    (\mu^*,u^*,0) \neq (\mu_\lambda, \theta_\lambda,0).
\end{equation}
Let us show that
$$\theta_\lambda \neq u^*.$$
Arguing by contradiction, if $\theta_\lambda = u^*$, it follows from (\ref{32147}) that $\mu_\lambda = \mu^*$, which is a contradiction with (\ref{654789}). Therefore, alternative 3 of Theorem~\ref{Thsistema} occurs with
$$\psi_\lambda :=u^*.$$

\underline{Case 3: $u^*=v^*=0$:}

Define
$$\widehat{u}_n:=\frac{u_n}{\|u_n\|_0} \quad\mbox{and} \quad \widehat{v}_n:=\frac{v_n}{\|v_n\|_0}, \quad n\geq 1.$$
Then the same argument as above shows that, up to a subsequence if necessary,
$$(\widehat{u}_n,\widehat{v}_n) \rightarrow (z,w) \quad \mbox{in}~\mathcal{C}_0^1(\overline{\Omega})\times \mathcal{C}_0^1(\overline{\Omega}),$$
with $\|z\|_0=\|w\|_0 = 1$ and $z,w>0$. It follows from (\ref{F}) and $\mathfrak{F}(\mu_n,u_n,v_n)=0$ that  $(\mu_n,\widehat{u}_n,\widehat{v}_n)$ satisfies
\begin{equation*}
    \left\{\begin{array}{l}
          \widehat{u}_n=    (-\Delta)^{-1}\left[\displaystyle\frac{1}{\|u_n\|_0}\left(\frac{RM-SN}{PR-QS}\right)(\mu_n,u_n,v_n)\right],\vspace{0.5cm} \\
          \widehat{v}_n =  (-\Delta + Z)^{-1} \left[\displaystyle\frac{1}{\|v_n\|_0}\left(\frac{PN-QM}{PR-QS}\right)(\mu_n,u_n,v_n) + Z(\widehat{v}_n)\right].
    \end{array}\right.
\end{equation*}
Letting $n \rightarrow \infty$ yieds
\begin{equation*}
    \left\{\begin{array}{l}
          z=    (-\Delta)^{-1}\left[\displaystyle\left(\frac{RM_u-S_uzN}{PR}\right)(\mu^*,0,0)\right],\vspace{0.5cm} \\
          w =  (-\Delta + Z)^{-1} \left[\displaystyle\left(\frac{PN_vw-Q_vwM}{PR}\right)(\mu^*,0,0) + Z(w)\right].
    \end{array}\right.
\end{equation*}
By elliptic regularity, $z, w \in W_0^{2,p}(\Omega)$ for all $p>1$. Moreover, the above equalites are equivalent to 
\begin{equation*}
    \left\{ \begin{array}{ll}
         -\mbox{div}(P(0,0) \nabla z) =\lambda a(x) z &\mbox{in}~\Omega,\\
         -\mbox{div}(R(0,0) \nabla w) = \mu^* b(x) w&\mbox{in}~\Omega,\\
         z=w=0 &\mbox{on} ~\partial \Omega.
    \end{array} \right.
\end{equation*}
Consequently,
\begin{eqnarray*}
\lambda &=& \sigma_1[-\mbox{div}(P(0,0) \nabla);a],\\
\mu^* &=& \sigma_1[-\mbox{div}(R(0,0) \nabla);b]
\end{eqnarray*}
and alternative 4 is satisfied. The proof is complete.
\end{demThs}

\section{Applications}\label{sec5}

In this section we apply Theorems~\ref{Thsistema} and~\ref{Thsistema'} to some particular systems in order to obtain conditions on parameters $\lambda,\mu \in \R$ which guarantee existence of coexistence states. 

First, once that a necessary condition to be able to use Theorems~\ref{Thsistema} and~\ref{Thsistema'} is that a semitrivial solution (\ref{stl}) (or (\ref{stm})) is nondegenerate,  we present the following auxiliary result.

\begin{Pro}\label{Propdeg}
Consider the equation
\begin{eqnarray}\label{semitrivial}
\left\{ \begin{array}{ll}
     -\mbox{div}(d(w)\nabla w) = \gamma  w + h(w)w&\mbox{in}~\Omega,  \\
     w=0&\mbox{on}~\partial\Omega, 
\end{array}\right.
\end{eqnarray}
under the following assumptions:
\begin{enumerate}
    \item[($H_d$)] $d:[0,\infty) \rightarrow [0,\infty)$ is a function of class $\mathcal{C}^2$, non-decreasing and  there exists a positive constant $d_0$ such that
    $$d(s) \geq d_0 \quad \forall s \in [0,\infty).$$
    \item[($H_h$)] $h: \R \rightarrow \R$ is a continuous function.
\end{enumerate}
If $h'(w)<0$ for all $w \geq 0$, then any strong solution $w \in \mbox{int}({\cal P})$ of (\ref{semitrivial}) is nondegenerate. 
\end{Pro}
\begin{proof}
Let $w_0 \in W^{2,p}(\Omega) \cap \mbox{int}({\cal P})$ be a strong solution of (\ref{semitrivial}). Performing the change of variables
$$ I(s):=\int_0^sd(t)dt \quad s \geq 0,$$
(\ref{semitrivial}) is rewritten as
\begin{eqnarray}\label{semitrivialw0}
\left\{ \begin{array}{ll}
     -\Delta[I(w_0)] = \gamma  w_0 + h(w_0)w_0&\mbox{in}~\Omega,  \\
     w_0=0&\mbox{on}~\partial\Omega.
\end{array}\right.
\end{eqnarray}
Let $x_M \in \Omega$ such that $w_0(x_M) = \|w_0\|_0$. Since $s\mapsto d(s)$ is non-decreasing, we have $I(w_0(x_M)) = \|I(w_0)\|_0$. Therefore,
\begin{eqnarray*}
0&\leq&-\Delta[I(w_0(x_M))] = \gamma  w_0(x_M) + h(w_0(x_M))w_0(x_M)\\
0&\leq&\gamma + h(w_0(x_M)).
\end{eqnarray*}
By the monotonicity of $h(s)$, we obtain from above inequality that
\begin{equation}\label{positive}
0\leq\gamma + h(w_0(x)) \quad \forall x \in \Omega.    
\end{equation}
The linearization of (\ref{semitrivialw0}) at $(\gamma,w_0)$ is given by
\begin{eqnarray}\label{stlinear}
\left\{ \begin{array}{ll}
     -\Delta[d(w_0)\xi]= [\gamma  +h(w_0)+w_0h'(w_0) ]\xi &\mbox{in}~\Omega,  \\
     \xi=0&\mbox{on}~\partial\Omega, 
\end{array}\right.
\end{eqnarray}
Suppose by contradiction that $\xi \not\equiv0$ is a strong solution of (\ref{stlinear}). Using the change of variable
$$d(w_0)\xi = \psi,$$
(\ref{stlinear}) is equivalent to
\begin{eqnarray*}
\left\{ \begin{array}{ll}
     -\Delta \psi = \displaystyle\frac{\gamma +h(w_0)+w_0h'(w_0) }{d(w_0)}\psi &\mbox{in}~\Omega,  \\
     \xi=0&\mbox{on}~\partial\Omega, 
\end{array}\right.
\end{eqnarray*}
which implies that
$$0=\sigma_j\left[-\Delta - \frac{\gamma +h(w_0)+w_0h'(w_0) }{d(w_0)}\right],\quad\mbox{for some $j\geq 1$.}
$$
From the dominance property of the principal eigenvalue it follows that
\begin{equation}\label{5ok}
0\geq\sigma_1\left[-\Delta - \frac{\gamma +h(w_0)+w_0h'(w_0) }{d(w_0)}\right].
\end{equation}
On the other hand, since $w_0\in \mbox{int}({\cal P})$ is a solution of (\ref{semitrivialw0}), then $I(w_0(x))>0$ for all $x \in \Omega$ and, hence,

\begin{equation}\label{95d}
0= \sigma_1\left[-\Delta - \frac{\gamma  w_0 + h(w_0)w_0}{I(w_0)}\right].
\end{equation}
We claim that 
\begin{equation}\label{claim}
- \frac{\gamma  w_0 + h(w_0)w_0}{I(w_0)} <  - \frac{\gamma  +h(w_0)+w_0h'(w_0) }{d(w_0)} .
\end{equation}
Assume this claim for a moment. By the monotonicity properties of the principal eigenvalue,  (\ref{claim}) and (\ref{95d}) imply   that
$$0< \sigma_1\left[-\Delta - \frac{\gamma  +h(w_0)+w_0h'(w_0) }{d(w_0)}\right],$$
which is a contradiction with (\ref{5ok}). Hence, to complete the proof it suffices to show (\ref{claim}). Indeed, (\ref{claim}) is equivalent to
\begin{eqnarray}\label{dddd}
\frac{(\gamma  + h(w_0))w_0}{I(w_0)} &>&   \frac{\gamma  +h(w_0)}{d(w_0)} +\frac{w_0h'(w_0)}{d(w_0)} 
\end{eqnarray}
On the other hand,  since $s \mapsto d(s)$ is non-decreasing, we have that
\begin{equation}\label{dw0}
0<I(w_0)= \int_0^{w_0}d(t)dt \leq d(w_0)\int_0^{w_0}dt=d(w_0)w_0.
\end{equation}
Thus, it follows from (\ref{positive}) and (\ref{dw0}) that
\begin{equation}
    \frac{(\gamma  + h(w_0))w_0}{I(w_0)} >   \frac{\gamma  +h(w_0)}{d(w_0)}.
\end{equation}
Moreover, once that $h'(s)<0$, $s \geq 0$, we can infer that $0>w_0h'(w_0)/d(w_0)$.
Combining this inequality with (\ref{dw0}), we obtain (\ref{dddd}). This completes the proof.
\end{proof}

From the point of view of population dynamics, an important particular case of (\ref{semitrivial}) is the logistic equation,  that is,
\begin{eqnarray}\label{logistico}
\left\{ \begin{array}{ll}
     -\mbox{div}(d(w)\nabla w) = \gamma  w - w^2&\mbox{in}~\Omega,  \\
     w=0&\mbox{on}~\partial\Omega.
\end{array}\right.
\end{eqnarray}
It appears, for instance, when one considers the Lotka-Volterra, Holling-II or Holling-Tanner reaction term, which are most commonly used in the literature.

The following result is consequence of, for instance, \cite{CC1991, ACP} and Proposition~\ref{Propdeg}.
\begin{Lema}
\label{obs2}
Under the hypothesis ($H_d$), (\ref{logistico}) possesses a  positive (classical) solution if, and only if, 
$$
\gamma >\sigma_1[-d(0)\Delta]
$$
 and it is unique if exists. Moreover, it is non-degenerate. 
\end{Lema}

\subsection{An abstract model}

To illustrate how one can use Theorems~\ref{Thsistema} and~\ref{Thsistema'} to determine a region of coexistence of positive solutions, we will consider the following system:
\begin{eqnarray}\label{Ap1}
\left\{ \begin{array}{ll}
     -\mbox{div}(A(v)G'(u)H(v)\nabla u + A(v)G(u)H'(v) \nabla v)) = u(\lambda - u - bv)&\mbox{in}~\Omega,  \\
     -\Delta v = v(\mu - v + cu)&\mbox{in}~\Omega,\\
     u=v=0&\mbox{on}~\partial\Omega, 
\end{array}\right.
\end{eqnarray}
where $b,c$ are positive constants. Here, $u$ and $v$ represent the population densities of a prey and a predator, respectively, inhabiting in $\Omega$. In the case of the prey, 
a term of self-diffusion and another of cross-diffusion appear. We also assume:
\begin{enumerate}
    \item[($H_A$)] $A:[0,\infty) \rightarrow [0,\infty)$ are nontrivial functions of class $\mathcal{C}^2$ such that
    $$
    A(v) \geq \underline{A}>0 \quad \forall v \geq 0
    $$
    for some positive constant $\underline{A}$.
    \item[($H_G$)] $G:[0,\infty) \rightarrow [0,\infty)$ is a function of class $\mathcal{C}^3$ such that
    $$
    G(0)=0, \quad \lim_{s \rightarrow\infty}G(s)= \infty
    $$
    and $G'$ is a non-decreasing function such that
    $$
    G'(s) \geq G_0>0 \quad \forall s \geq 0,
    $$
    for some positive constant $G_0$.
    
    \item[($H_H$)] $H:[0,\infty) \rightarrow \R $ is a function of class $\mathcal{C}^2$  such $$
    H_0 \leq H(v) \leq H_1 \quad \forall v \geq 0
    $$
    for some positive constants $H_0$ and $H_1$.
\end{enumerate}
The following functions satisfy all the above hypothesis: 
$$A(v) = v+1,\quad G(u) =u^2+u,\quad H(v)=\frac{v+2}{v+1}.$$

Let us show that, combining Theorems~\ref{Thsistema} and~\ref{Thsistema'} with, for instance, a result of a priori bound and an appropriate non-existence result, one can determine a region of coexistence states.

The non-negative semitrivial solutions $(u,0)$ and $(0,v)$   of (\ref{Ap1}) are given by
\begin{eqnarray}\label{stl2}
\left\{\begin{array}{ll}
    -\mbox{div}(A(0)H(0) G'(u) \nabla u) = u (\lambda - u) & \mbox{in }\Omega,  \\
    u=0 & \mbox{on }\partial\Omega,
\end{array} \right.
\end{eqnarray}
and
\begin{eqnarray}\label{stm2}
\left\{\begin{array}{ll}
    -\Delta v = v (\mu - v) & \mbox{in }\Omega,  \\
    v=0 & \mbox{on }\partial\Omega,
\end{array} \right.
\end{eqnarray}
respectively. Since $s \mapsto G'(s)$ is a non-decreasing function, by Lemma~\ref{obs2}, (\ref{stl2}) and (\ref{stm2}) possess a (unique and nondegenerate) positive solution if, and only if, 
$$\lambda> \sigma_1[-A(0)H(0)G'(0) \Delta]\quad\mbox{and}\quad\mu > \lambda_1
$$ 
and they will be denoted by $\theta_\lambda$ and $\theta_\mu$, respectively. 

In addition, the maps $\lambda \in (\sigma_1[-A(0)H(0)G'(0) \Delta],\infty) \mapsto \theta_\lambda \in \mathcal{C}_0^1(\overline{\Omega})$ and $\mu \in  (\lambda_1,\infty) \mapsto \theta_\mu \in \mathcal{C}_0^1(\overline{\Omega})$ are increasing.

Moreover, for this system,  the eigenvalues that appear in Theorems~\ref{Thsistema} and~\ref{Thsistema'} can be defined as follow:
\begin{equation}\label{muAp1}
\mu_\lambda:= \left\{ \begin{array}{ll}
    \sigma_1 [-\Delta - c \theta_\lambda]     & \mbox{if}~ \lambda >  \sigma_1[-A(0)H(0)G'(0) \Delta] , \\
     \lambda_1& \mbox{if}~\lambda \leq  \sigma_1[-A(0)H(0)G'(0) \Delta],
\end{array} \right.
\end{equation}
and
\begin{equation}\label{lAp1}
    \lambda_\mu:= \left\{ \begin{array}{ll}
         \sigma_1 [-div(G'(0)A(\theta_\mu)H(\theta_\mu)e^{-h(\theta_\mu)}\nabla) + c \theta_\mu e^{-h(\theta_\mu)};e^{-h(\theta_\mu)}] & \mbox{if}~\mu >  \lambda_1,  \\
         0&\mbox{if}~\mu \leq\lambda_1,
    \end{array} \right.
\end{equation}
where 
$$
h(z):= \int_0^z\frac{H'(s)}{H(s)}ds, \quad z \geq 0.
$$

It should be noted that, by the monotonicity properties of the principal eigenvalue, the function $\lambda \mapsto \mu_\lambda$ is  decreasing for 
$\lambda>\sigma_1[-A(0)H(0)G'(0) \Delta]$. However, it is not easy to ascertain monotony properties of the map $\mu\mapsto\lambda_\mu$.

Now we will show a result of a priori bound on the coexistence states of (\ref{Ap1}).

\begin{Lema}\label{cotasAp}
Suppose that $(u,v)$ is a coexistence state of (\ref{Ap1}). Then there exists a positive constant $\overline{C}=\overline{C}(\lambda,b)$  such that
$$
\|u\|_0 \leq \overline{C}\quad \mbox{and}\quad\|v\|_0 \leq \mu+c\overline{C}.
$$
Moreover, there exists a positive constant $C^*=C^*(\lambda,\mu,b,c,G_0,H_0,H_1)$ such that
$$
\|u\|_{\mathcal{C}^1},~\|u\|_{\mathcal{C}^1} \leq C^*.
$$
\end{Lema}
\begin{proof}
Let $(u,v)$ be a coexistence state of (\ref{Ap1}). First, we will get an estimate on $\|u\|_0$. Once that
$$
A(v)G'(u)H(v)\nabla u + A(v)G(u)H'(v) \nabla v) = A(v)\left(  \nabla (G(u) H(v))\right),
$$
$(u,v)$ satisfies
\begin{eqnarray}\label{cota01}
\left\{\begin{array}{ll}
     -\mbox{div}\left[A(v)\left(  \nabla (G(u) H(v))\right) \right] = u(\lambda - u - bv)&  \mbox{in }\Omega, \\
     u=v=0& \mbox{on }\partial \Omega. 
\end{array} \right.
\end{eqnarray}
Multiplying (\ref{cota01}) by $[G(u)H(v)-G(\lambda)H_1]_+$ and applying the formula of integration by parts gives
\begin{eqnarray}\label{cota02}
0 \leq \int_\Omega A(v) \left|\nabla \left[G(u) H(v) - G(\lambda)H_1\right]_+\right|^2 &=& \int_\Omega u(\lambda - u - bv)\left[G(u) H(v) - G(\lambda)H_1\right]_+ \nonumber\\
&\leq&\int_\Omega u(\lambda - u)\left[G(u) H(v) - G(\lambda)H_1\right]_+.
\end{eqnarray}
Note that in $\{x \in \Omega;~G(u) H(v) \leq G(\lambda)H_1\}$, we have:
\begin{equation}\label{ul1}
u(\lambda - u)\left[G(u) H(v) - G(\lambda)H_1\right]_+=0 
\end{equation}
and in $\{x \in \Omega;~G(u) H(v) > G(\lambda)H_1\}$, we get $u>\lambda$ and, hence,
\begin{equation}\label{ul2}
u(\lambda - u)\left[G(u) H(v) - G(\lambda)H_1\right]_+< 0.
\end{equation} 
Combining (\ref{cota02}), (\ref{ul1}) and (\ref{ul2}), we can infer that
\begin{equation*}
    0=\int_\Omega u(\lambda - u)\left[G(u) H(v) - G(\lambda)H_1\right]_+.
\end{equation*}
Since the function $u(\lambda - u)\left[G(u) H(v) - G(\lambda)H_1\right]_+$ is non-negative (according to (\ref{ul1})-(\ref{ul2})), the above equality implies that
$$
u(\lambda - u)\left[G(u) H(v) - G(\lambda)H_1\right]_+ \equiv 0 \quad \mbox{in}~\Omega.
$$
In view of (\ref{ul1}) and (\ref{ul2}), we must have, necessarily, $G(u) H(v) \leq G(\lambda)H_1$ and, hence,
$$
u \leq G^{-1}\left(G(\lambda) H_1/H_0 \right)=: \overline{C}(\lambda).
$$

Now, we will show a priori bound on  $\|v\|_0$ and $\|v\|_{\mathcal{C}^1}$. 
Let $x_M \in \Omega$ such that $v_M:=v(x_M) = \max_\Omega v(x)$. Then,
\begin{eqnarray*}
0&\leq& -\Delta v_M = v_M(\mu - v_M + cu(x_M)) \\
v_M &\leq& \mu + c u(x_M) \leq \mu + c \overline{C}(\lambda)
\end{eqnarray*}
Then, $\|v(\mu-v+cu)\|_p$ is bounded for all $p>1$, and hence, by elliptic regularity, there exists a constant $C_0 = C_0(\lambda, c, \mu)$ such that
$$
\|v\|_{\mathcal{C}^1} \leq C_0.
$$

To complete the proof, it remains to show that $\|u\|_{\mathcal{C}^1}$ is also bounded. Indeed, since $(u,v)$ verifies (\ref{cota01}) and $v$ is bounded in 
$\mathcal{C}^1(\overline\Omega)$, we can apply again the standard elliptic regularity to (\ref{cota01}) and conclude that there exists a positive constant $C_1=C_1(\lambda,\mu,b,c)$ such that
\begin{equation}\label{C1}
\|G(u)H(v)\|_{\mathcal{C}^1} \leq C_1.
\end{equation}
On the other hand,
\begin{eqnarray*}
    &\nabla(G(u)H(v)) = G'(u)H(v)\nabla u + G(u) H'(v) \nabla v& \\
&|\nabla(G(u)H(v)) - G(u) H'(v) \nabla v | = |G'(u)H(v)\nabla u| &
\end{eqnarray*}
Using the triangular inequality, ($H_G$) and ($H_H$) we obtain that
\begin{equation}\label{Cfinal}
    |\nabla(G(u)H(v))| + |G(u) H'(v) \nabla v | \geq H_0G_0 |\nabla u| 
\end{equation}
Since $u$  is bounded  in $\mathcal{C}(\overline{\Omega})$ and $v$ is bounded in $\mathcal{C}^1(\overline{\Omega})$, there exists a positive constant $C_2 =C_2(\lambda,\mu,b,c)$ such that
\begin{equation}\label{C2}
\|G(u)H'(v)\nabla v\|_0 \leq C_2
\end{equation}
Thus, in view of (\ref{C1}) and (\ref{C2}), we can infer from (\ref{Cfinal}) that
$$
C_1 + C_2 \geq G_0 H_0 \|\nabla u\|_0. 
$$
This completes the proof.
\end{proof}

The next lemma gives an appropriate non-existence result of positive solutions.

\begin{Lema}\label{NexistenciaAp1}
\begin{enumerate}
    \item[(a)] If $\lambda \leq 0$, then (\ref{Ap1}) does not admit coexistence states.
    \item[(b)] The problem  (\ref{Ap1}) does not admit coexistence states for $\mu \leq \lambda_1 - c \overline{C}$, where $\overline{C} = \overline{C}(\lambda,b)$ is the positive constant given in Lemma \ref{cotasAp}.
\end{enumerate}
\end{Lema}
\begin{proof}
Let $(u,v)$ be a coexistence state of (\ref{Ap1}). Suppose by contradiction that $\lambda \leq 0$. Then, by (\ref{cota01}) we get
\begin{eqnarray*}
\left\{\begin{array}{ll}
     -\mbox{div}\left[A(v) \nabla \left(G(u)H(v)\right) \right] \leq 0&  \mbox{in }\Omega, \\
     G(u) H(v)=0& \mbox{on }\partial \Omega. 
\end{array} \right.
\end{eqnarray*}
It follows from the Maximum Principle that $G(u) H(v) \leq 0$ and, hence, $G(u)\leq 0$, which is impossible since $G(s) > 0$ for $s\geq 0$ (according to ($H_G$)).

To prove (b), note that $v$  verifies
\begin{eqnarray*}
\left\{\begin{array}{ll}
     -\Delta v  = v(\mu - v + cu)&  \mbox{in }\Omega, \\
    v=0& \mbox{on }\partial \Omega.
\end{array} \right.
\end{eqnarray*}
Consequently,
$$
\mu = \sigma_1[-\Delta + v - cu].
$$
By the monotonicity properties of the principal eigenvalue combined with Lemma \ref{cotasAp}, we find that
$$
\mu > \sigma_1[-\Delta  - c\overline{C}] = \lambda_1 - c \overline{C},
$$
which completes the proof.
\end{proof}

Now, we are able to apply Theorems~\ref{Thsistema} and~\ref{Thsistema'} to obtain a region of coexistence of (\ref{Ap1}).

\begin{Teo}\label{thap1}
Assume ($H_A$), ($H_G$) and ($H_H$). Then
(\ref{Ap1}) possesses at least one coexistence state for each $(\lambda,\mu) \in \R^2$ such that 
\begin{equation}
\label{condicoe1}
\mu > \mu_\lambda\quad\mbox{and}\quad\lambda> \lambda_\mu, 
\end{equation}
where $\mu_\lambda$ and $\lambda_\mu$ are given in (\ref{muAp1}) and (\ref{lAp1}).
\end{Teo}
\begin{proof}
First, note that by ($H_A$), ($H_G$) and ($H_H$), all hypothesis of Theorems~\ref{Thsistema} and~\ref{Thsistema'} are satisfied.

Now fix $\mu > \lambda_1$. By Theorem~\ref{Thsistema'}, from the point 
$$
(\lambda,u,v) = (\lambda_\mu, 0, \theta_\mu)
$$
emanates a continuum $\mathfrak{C} \subset \R \times \mbox{int}({\cal P}) \times\mbox{int}({\cal P})$ of coexistence states of (\ref{Ap1}) and one of the alternatives of Theorem~\ref{Thsistema'} occurs. By uniqueness of positive solution of (\ref{stm2}), alternative 3 cannot be satisfied. Moreover, since  $\mu > \lambda_1$, alternative 4 also cannot occur. We will show that alternative 2 is not true and to this end we will proceed by contradiction. Otherwise, there exists a positive solution of (\ref{stl2}), $(\lambda^*,\theta_{\lambda^*}$), such that
$$
\mu= \mu_{\lambda^*}= \sigma_1 [-\Delta - c \theta_\lambda^*]. 
$$
However, since the map $\lambda \in (\lambda_1,\infty) \mapsto \sigma_1 [-\Delta - c \theta_\lambda]$ is decreasing, we can infer from above equality that
$$
\mu = \mu_{\lambda^*} \leq \sigma_1[-\Delta] =\lambda_1,
$$
which is a contradiction with the initial assumption $\mu > \lambda_1$. 

Therefore, $\mathfrak{C}$ is unbounded in $\R \times \mathcal{C}_0^1(\overline{\Omega}) \times \mathcal{C}_0^1(\overline{\Omega})$. Once that the coexistence states of (\ref{Ap1}) are bounded in $\mathcal{C}_0^1(\overline{\Omega}) \times \mathcal{C}_0^1(\overline{\Omega})$ (see Lemma~\ref{cotasAp}), then  $\mbox{Proj}_\R \mathfrak{C}$ is unbounded. Finally, it follows from Lemma~\ref{NexistenciaAp1} that it extends to infinity in positive values of $\lambda$. By global nature of $\mathfrak{C}$, it follows that $(\lambda_\mu,\infty) \subset \mbox{Proj}_\R \mathfrak{C} $. Since $\mu > \lambda_1$ is arbitrary, we obtain that (\ref{Ap2}) possesses at least one coexistence states for all $(\lambda,\mu) \in \R^2$ such that $\mu > \lambda_1$ and $\lambda> \lambda_\mu$.

Now, fix $\lambda > \sigma_1[-A(0)H(0)G'(0) \Delta]$. By Theorem~\ref{Thsistema}, from the point 
$$ 
(\mu,u,v) = (\mu_\lambda, \theta_\lambda,0)
$$
emanates a continuum $\mathfrak{C}' \subset \R\times\mbox{int}({\cal P}) \times \mbox{int}({\cal P})$ of coexistence states of (\ref{Ap1}) and one of the alternatives of Theorem~\ref{Thsistema} occurs. Arguing as above, alternatives 3 and 4 cannot be satisfied. Hence, it happens or alternative 1 or 2.

Suppose that alternative 1 holds. Then $\mathfrak{C}'$ is unbounded in $\R\times\mbox{int}({\cal P}) \times \mbox{int}({\cal P})$. Since the coexistence states of (\ref{Ap1}) are bounded in $\mathcal{C}_0^1(\overline{\Omega}) \times \mathcal{C}_0^1(\overline{\Omega})$, $\mbox{Proj}_\R \mathfrak{C}'$ must be unbounded in $\mu \in \R$. Since (\ref{Ap1}) does not have coexistence states for $\mu$ small (according to Lemma~\ref{NexistenciaAp1} (b)), by global nature of $\mathfrak{C}'$, we find that
$$
(\mu_\lambda, \infty) \subset \mbox{Proj}_\R \mathfrak{C}'.
$$
In particular, $(\mu_\lambda, \lambda_1] \subset \mbox{Proj}_\R \mathfrak{C}'$. This result combined with the coexistence region obtained above  prove the existence of coexistence states of (\ref{Ap1}) for all $(\lambda,\mu) \in \R^2$ such that $\mu > \mu_\lambda$ and $\lambda> \lambda_\mu$.

Suppose now that alternative 2 holds. Then there exists a positive solution $(\mu^*,\theta_{\mu^*})$ of (\ref{stm2}) such that
$$
\lambda = \lambda_{\mu^*} = \sigma_1 [-div(G'(0)A(\theta_{\mu^*})H(\theta_{\mu^*})e^{-h(\theta_{\mu^*})}\nabla) + c \theta_{\mu^*} e^{-h(\theta_{\mu^*})};e^{-h(\theta_{\mu^*})}],
$$
and $(\mu^*, 0, \theta_{\mu^*}) \in \mathfrak{C}'$. On the other hand, since (\ref{stm2}) admits a positive solution if, and only if $\mu > \lambda_1$, then $\mu^* >\lambda_1$. Hence, by global nature of $\mathfrak{C}'$, it becomes apparent that
$$
(\mu_\lambda, \lambda_1] \subset (\mu_\lambda, \mu^*) \subset \mbox{Proj}_\R \mathfrak{C}',
$$
and again we obtain the result.
\end{proof}
Figure \ref{fig1} illustrates a possible region of coexistence of (\ref{Ap1}) given by condition (\ref{condicoe1}) of Theorem~\ref{thap1}.

\begin{figure}
\centering
\includegraphics[scale=0.8]{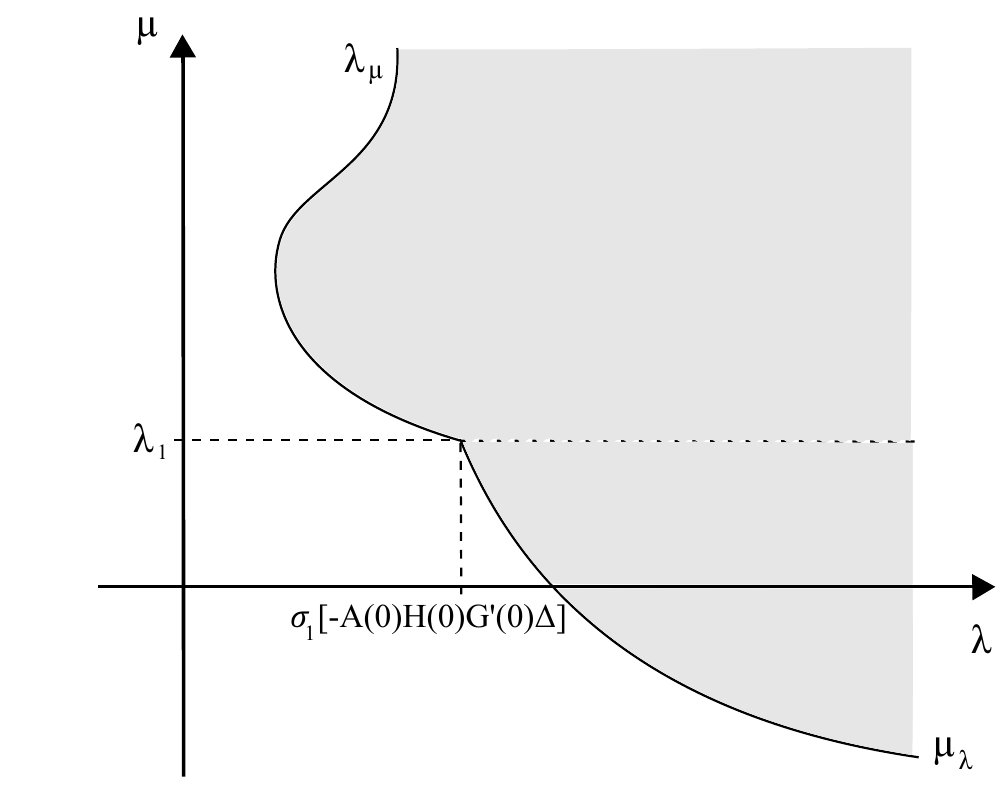}
\caption{\label{fig1} A region of coexistence of (\ref{Ap1}).}
\end{figure}

\subsection{A chemotaxis model}
We now apply the bifurcation Theorems~\ref{Thsistema} and~\ref{Thsistema'} to study the following chemotaxis model:
\begin{equation}\label{Ap2}
    \left\{ \begin{array}{ll}
         -\Delta u + \mbox{div}(\chi f(v) u \nabla v) = u(\lambda - u +bv)& \mbox{in}~\Omega,\\
         -\Delta v= v(\mu - v - cu)& \mbox{in}~\Omega,\\
         u=v=0 & \mbox{on}~\partial \Omega,
    \end{array} \right.
\end{equation}
where $f:[0,\infty) \rightarrow [0,\infty)$ is a function of class $\mathcal{C}^2$, $\chi$, $c$ are positive constants and $b \in \R$. Although chemotaxis models have been extensively studied in recent years (not so much the stationary models), we refer to \cite{fujie}, \cite{stinner} and \cite{winkler} and references therein, where a nonlinear sensitivity term is included. See also \cite{tello} for a chemotaxis model with competition interaction. 

We emphasize that (\ref{Ap2}) is not included in the hypotheses of the abstract model (\ref{Ap1}).

To begin our analysis, note that the semitrivial solutions $(u,0)$ and $(0,v)$ of (\ref{Ap2}) are the positive solutions of the logistic equation (\ref{logistico}) with $d \equiv 1$ and, as discussed in Lemma~\ref{obs2}, it possesses a (unique and non-degenerate) positive solution if, and only if, $\gamma> \lambda_1$ and it will be denoted by $\theta_\gamma$. In (\ref{Ap2}), the eigenvalues given by Theorems~\ref{Thsistema} and~\ref{Thsistema'} are:
\begin{equation}\label{ml2}
    \mu_\lambda := \sigma_1[-\Delta + c \theta_\lambda], \quad \lambda > \lambda_1
\end{equation}
and
\begin{equation}\label{lm2}
    \lambda_\mu := \sigma_1[-\mbox{div}(e^{\chi F( \theta_\mu)} \nabla) - b \theta_\mu e^{\chi F( \theta_\mu)}; e^{\chi F( \theta_\mu)}], \quad \mu > \lambda_1,
\end{equation}
where $ F(z):= \int_0^zf(s)ds$. We extend these functions by $\lambda_\mu:=\lambda_1$ and $\mu_\lambda:=\lambda_1$, for each $\mu,\lambda \leq \lambda_1$.

Note that, since $f(\cdot)$ is a non-negative function, we have that $F(\cdot)$ is non-decreasing.

The next lemma gives us a result of a priori bounded of the coexistence states of (\ref{Ap2}).
\begin{Lema}\label{cotasAp2}
Suppose that $(u,v)$ is a coexistence state of (\ref{Ap2}) with $\mu >\lambda_1$. Then, there exists a positive constant $C = C(\lambda,\mu,b,c,\chi)$ such that
$$
\|u\|_{\mathcal{C}^1},~\|v\|_{\mathcal{C}^1} \leq C.
$$
\end{Lema}
\begin{proof}
Let $(u,v)$ a coexistence state of (\ref{Ap2}). Then $v$ satisfies
$$-\Delta v = v(\mu - v -cu) \leq v(\mu- v), \quad \mbox{in}~ \Omega.$$
Thus, $v$ is a positive subsolution of (\ref{logistico}) with $\gamma = \mu$ and $d \equiv 1$, whose the unique solution is $\theta_\mu$. By the uniqueness of positive solutions,
$$
v \leq \theta_\mu \leq \mu.
$$

On the other hand, the first equation of (\ref{Ap2}) can be re-written as
\begin{equation*}
    \left\{ \begin{array}{ll}
         -\mbox{div}(e^{\chi F(v)} \nabla (ue^{-\chi F(v)})) = u(\lambda - u +bv)& \mbox{in}~\Omega,\\
          u=0 & \mbox{on}~\partial \Omega.
    \end{array} \right.
\end{equation*}
By adapting the proof of  Lemma~\ref{cotasAp}, one can obtain the  existence of a positive constant $C_0 = C_0(\lambda,\mu,b)$ such that
$$\|u\|_0 \leq C_0.$$
Finally, by elliptic regularity, there exists a positive constant $C= C(\lambda,\mu,b,c,\chi)$ such that
$$ \|u\|_{\mathcal{C}^1},~\|v\|_{\mathcal{C}^1} \leq C.$$
\end{proof}

The next result establishes the nonexistence of coexistence states for certain values of $\lambda$ and $\mu$, including for $\lambda$ large.

\begin{Lema}\label{NexistenciaAp2}
\begin{enumerate}
    \item[(a)] If $\mu\leq \lambda_1$, then (\ref{Ap2}) does not admit coexistence states.
    \item[(b)] Fix $\mu> \lambda_1$. Then, there exists a constant $C=C(\mu)$ such that (\ref{Ap2}) does not admit coexistence states for $\lambda < C(\mu)$.
    \item[(c)] Fix $\mu>\lambda_1$. Then, there exists a constant $\Lambda_0= \Lambda_0(\mu)>0$  such that (\ref{Ap2}) does not admit coexistence states for $\lambda>\Lambda_0(\mu)$.
\end{enumerate}
\end{Lema}
\begin{proof}
Let $(u,v)$ be a coexistence state of (\ref{Ap2}). Then $v$ satisfies 
$$
-\Delta v + (v+cu)v = \mu v \quad \mbox{in}~\Omega, \quad v=0 \quad \mbox{on}~\partial\Omega.
$$
Consequently,
$$\mu= \sigma_1[-\Delta +v +cu].$$
By the monotonicity properties of the principal eigenvalue, we find that $\mu > \sigma_1[-\Delta]= \lambda_1$, proving the paragraph (a).

Now, we will prove (b). Performing the change of variable $w =ue^{-\chi F(v)}$ in the first equation of (\ref{Ap2}), we obtain that the positive function $w$ satisfies
\begin{equation}\label{eq:W}
    -\mbox{div}(e^{\chi F(v)} \nabla w)  +  (w  e^{\chi F(v)}-bv) w e^{\chi F(v)} =  \lambda w e^{\chi F(v)}, \quad\mbox{in}~\Omega, \quad w=0\quad\mbox{on}~\partial\Omega.
\end{equation}
Thus, 
$$
\lambda = \sigma_1[-\mbox{div}(e^{\chi F(v)} \nabla)  +  (w  e^{\chi F(v)}-bv) e^{\chi F(v)};e^{\chi F(v)}].
$$
If $b\leq0$, combining the monotonicity properties of the principal eigenvalue with a priori bound $v \leq \mu$ and $1\leq e^{\chi F(v)} \leq e^{\chi F(\mu)}$ , we can infer that
$$
\lambda >  \sigma_1[-\Delta;e^{\chi F(v)}]>
\sigma_1[-\Delta;e^{\chi F(\mu)}] = \lambda_1e^{-\chi F( \mu)} =:C_1(\mu),
$$
since $\lambda_1=\sigma_1[-\Delta]>0$. On the other hand, if $b>0$ we have
$$
   \lambda > \sigma_1[-\Delta-b \mu e^{\chi F(\mu)};e^{\chi F(v)}].
$$
Now, we recall that the monotonicity of $\sigma_1[-\Delta-b \mu e^{\chi F(\mu)};e^{\chi F(v)}]$ with respect to weight function $e^{\chi F(v)}$ depends on  $\text{sign }\sigma_1[-\Delta-b \mu e^{\chi F(\mu)}]$. Thus, it follows from above inequality that
$$
   \lambda > 
  \left\{ \begin{array}{ll}
         \sigma_1[-\Delta-b \mu e^{\chi F(\mu)};e^{\chi F(\mu)}] = \lambda_1e^{-\chi F(\mu)} - b\mu& \mbox{if}~\sigma_1[-\Delta-b \mu e^{\chi F(\mu)}]>0,   \\
        0& \mbox{if} ~\sigma_1[-\Delta-b \mu e^{\chi F(\mu)}]=0,\\
        \sigma_1[-\Delta-b \mu e^{\chi F(\mu)};1] = \lambda_1 - b\mu e^{\chi F(\mu)} & \mbox{if}~\sigma_1[-\Delta-b \mu e^{\chi F(\mu)}]<0.
   \end{array}\right.
$$

In all cases, there exists a constant $C=C(\mu)$ such that (\ref{Ap2}) does not  admit coexistence states for $\lambda \leq C(\mu)$.

Finally, to prove (c) we will adapt the proof of Proposition 6.5 of \cite{DS2008}. We argue by contradiction. Fix $\mu>\lambda_1$ and assume that   there exists a coexistence state for all $\lambda>0$ large enough. Using again the change of variable $w =ue^{-\chi F(v)}$ we obtain that $w$ satisfies (\ref{eq:W}). It follows from $1\leq e^{\chi F(v)} \leq e^{\chi F(\mu)}$ that
$$
-\mbox{div}(e^{\chi F(v)} \nabla w) \geq \lambda w - e^{2\chi F(\mu)}w^2 + K(b,\mu) w, \quad \mbox{in}~\Omega,
$$
where $K(b,\mu):= \min\{0,b \mu e^{\chi F(\mu)}\} \leq 0$. Therefore $w$ is a supersolution of
\begin{equation}\label{eq:z}
\left\{ \begin{array}{ll}
     -\mbox{div}(e^{\chi F(v)} \nabla z) - K(b,\mu) z =  \lambda z - e^{2\chi F(\mu)}z^2& \mbox{in}~\Omega,  \\
     z=0&\mbox{on}~\partial\Omega. 
\end{array}  \right.
\end{equation}
It is well-known, see for instance \cite{DS2008}, that (\ref{eq:z}) has a (unique) positive solution for each $\lambda> \sigma_1[-\mbox{div}(e^{\chi F(v)} \nabla) - K(b,\mu)]$, say $z_\lambda$, such that
$$\frac{\lambda - \sigma_1[-\mbox{div}(e^{\chi F(v)} \nabla) - K(b,\mu)]}{e^{2\chi F(\mu)}\|\phi_\lambda\|_0} \phi_\lambda \leq z_\lambda,$$
where $\phi_\lambda$ stands for the positive eigenfunction associated to $\sigma_1[-\mbox{div}(e^{\chi F(v)} \nabla) - K(b,\mu)]$ with $|\phi_\lambda|_2 =1$. (We point out that $\phi_\lambda$ depends on $\lambda$ because $v$ depends on $\lambda$). Since $w$ is a supersolution of (\ref{eq:z}), by uniqueness,
$$
\frac{\lambda - \sigma_1[-\mbox{div}(e^{\chi F(v)} \nabla) - K(b,\mu)]}{e^{2\chi F(\mu)}\|\phi_\lambda\|_0} \phi_\lambda \leq w.
$$
On the other hand, by the monotonicity properties of the principal eigenvalue we have
$$
\sigma_1[-\mbox{div}(e^{\chi F(v)} \nabla) - K(b,\mu)] \leq \sigma_1[-\mbox{div}(e^{\chi F(\mu)} \nabla) - K(b,\mu)] :=s(\mu)
$$
and so, denoting 
$$
\tau(\lambda):= \frac{\lambda -s(\mu)}{e^{\chi F(\mu)}\|\phi_\lambda\|_0} 
$$
it holds
\begin{equation}\label{tau}
    \tau(\lambda) \phi_\lambda \leq w \leq u.
\end{equation}

We recall that $\|\phi_\lambda\|_0$ is uniform bound with respect to $\lambda$ (see, for instance, Theorem~4.1 in \cite{Stampacchia1965}). Then,
\begin{equation*}
    \tau(\lambda) \geq \frac{\lambda -s(\mu)}{Ce^{\chi F(\mu)}} \rightarrow \infty \quad \mbox{as}~\lambda \rightarrow \infty.
\end{equation*}
Now, since $v$ is a positive solution of the second equation of (\ref{Ap2}) and using (\ref{tau}), we get
\begin{equation}\label{uhh}
\mu= \sigma_1[-\Delta +v +cu] \geq \sigma_1[-\Delta +c \tau(\lambda)\phi_\lambda]=:g(\lambda).     
\end{equation}
Once that $\mu> \lambda_1$ is fixed,  to finish the proof is sufficient to show that
\begin{equation}\label{ginfinito}
    \lim_{\lambda \rightarrow \infty} g(\lambda) = +\infty,
\end{equation}
because (\ref{ginfinito}) produces a contradiction with (\ref{uhh}). In order to prove (\ref{ginfinito}) we argue by contradiction.
Observe that 
\begin{equation}\label{ginf}
    g(\lambda) = \inf_{\varphi \in H_0^1(\Omega)\setminus\{0\}} \frac{\int_\Omega |\nabla \varphi|^2 + c\tau(\lambda)\int_\Omega \phi_\lambda \varphi^2}{\int_\Omega \varphi^2}.
\end{equation}
Suppose otherwise that $g(\lambda)$ is bounded. Then, there exists a sequence $\varphi_\lambda \in H_0^1(\Omega)$ with $|\varphi_\lambda|_2=1$ and that attains the infimum (\ref{ginf}), that is,
\begin{equation}\label{geq}
    \int_\Omega |\nabla \varphi_\lambda|^2 + c\tau(\lambda)\int_\Omega \phi_\lambda \varphi_\lambda^2 =g(\lambda).
\end{equation}
Since $g(\lambda)$ is bounded, it follows from (\ref{geq}) that $\varphi_\lambda$ is bounded in $H_0^1(\Omega)$ and, up to a subsequence if necessary, there exists $\varphi_0 \geq 0$, $|\varphi_0|_2=1$ and $\varphi_0 \neq 0$, such that
\begin{equation}\label{varphil}
\varphi_\lambda \rightharpoonup \varphi_0 \quad\mbox{in}~H_0^1(\Omega), \quad \quad
    \varphi_\lambda \rightarrow \varphi_0 \quad\mbox{in}~L^2(\Omega) \quad\quad\mbox{as}~\lambda \rightarrow\infty.
\end{equation}
We study now $\phi_\lambda$. Combining the monotonicity properties of the principal eigenvalue with $1 \leq e^{\chi F(v)} \leq e^{\chi F(\mu)}$ yields
$$
\lambda_1=\sigma_1[-\Delta] \leq \sigma_1[-\mbox{div}(e^{\chi F(v)} \nabla) - K(b,\mu)] \leq \sigma_1[-\mbox{div}(e^{\chi F(\mu)} \nabla) - K(b,\mu)]. 
$$
Again since $\mu> \lambda_1$ is fixed, we can conclude that there exists $\sigma_0\geq\lambda_1>0$ such that (up to a subsequence if necessary)
$$
\sigma_1[-\mbox{div}(e^{\chi F(v)} \nabla) - K(b,\mu)] \rightarrow \sigma_0 \quad \mbox{as}~\lambda\rightarrow \infty.
$$
By definition,
\begin{equation}\label{phieq}
    -\mbox{div}(e^{\chi F(v)}\nabla \phi_\lambda) =\sigma_1[-\mbox{div}(e^{\chi F(v)} \nabla) - K(b,\mu)] \phi_\lambda \quad \mbox{in}~\Omega,
\end{equation}
and so,
$$
\int_\Omega|\nabla \phi_\lambda|^2 \leq \int_\Omega e^{\chi F(v)}|\nabla \phi_\lambda|^2 = \sigma_1[-\mbox{div}(e^{\chi F(v)} \nabla) - K(b,\mu)] \int_\Omega\phi_\lambda^2 \leq  \sigma_1[-\mbox{div}(e^{\chi F(\mu)} \nabla) - K(b,\mu)],
$$
whence we deduce that $\phi_\lambda$ is bounded in $H_0^1(\Omega)$ and, up to a subsequence if necessary, there exists $\phi_0 \geq 0$, $|\phi_0|_2=1$ and $\phi_0 \neq 0$, such that
\begin{equation}\label{phil}
\phi_\lambda \rightharpoonup \phi_0 \quad\mbox{in}~H_0^1(\Omega), \quad \quad
    \phi_\lambda \rightarrow \phi_0 \quad\mbox{in}~L^2(\Omega) \quad\quad\mbox{as}~\lambda \rightarrow\infty.
\end{equation}
Observe that (\ref{phieq}) is verified in $H^{-1}(\Omega)$, and so we can apply the homogenization technique (see, for instance, Theorem 2.1 in \cite{Kesavan}) and conclude that there exists a uniformly elliptic symmetric matrix $A \in (L^\infty(\Omega))^{N\times N}$ such that the following equation is verified in $H^{-1}(\Omega)$
$$
-\mbox{div}(A\nabla\phi_0) = \sigma_0 \phi_0.
$$
Since $\sigma_0 \phi_0 \geq 0$ and non-trivial, by the strong maximum principle $\phi_0>0$. Then, it follows from (\ref{geq}) that
$$
\limsup_{\lambda \rightarrow \infty} \int_\Omega\phi_\lambda \varphi_\lambda^2 = 0.
$$
In contrast, by (\ref{varphil}) and (\ref{phil}) we obtain that
$$
\limsup_{\lambda \rightarrow \infty} \int_\Omega\phi_\lambda \varphi_\lambda^2 = \int_\Omega \phi_0 \varphi_0^2>0,
$$
an absurdum. This completes the proof.
\end{proof}

We are ready to prove the main existence result:
\begin{Teo}\label{thap2}
Assume that $\mu>\lambda_1$ and let $\lambda_\mu$ and $\mu_\lambda$ be the functions defined by (\ref{lm2}) and (\ref{ml2}), respectively. Then if some of the following conditions are satisfied 
\begin{equation}
\label{final1}
\lambda > \lambda_\mu \quad\mbox{and} \quad \mu> \mu_\lambda
\end{equation}
or
\begin{equation}
\label{final2}
\lambda< \lambda_\mu \quad\mbox{and} \quad \mu < \mu_\lambda,
\end{equation}
then, there exists at least a coexistence state of (\ref{Ap2}).
\end{Teo}
\begin{proof}
As the reasoning is similar to the proof of  Theorem \ref{thap1}, we will be brief.

Fix $\mu > \lambda_1$. Since the semitrivial solutions of (\ref{Ap2}) are nondegenerate, by Theorem~\ref{Thsistema'}, from the point 
$$
(\lambda,u,v) = (\lambda_\mu, 0, \theta_\mu)
$$
emanates a continuum $\mathfrak{C} \subset \R \times \mbox{int}({\cal P}) \times\mbox{int}({\cal P})$ of coexistence states of (\ref{Ap2}) and one of the alternatives of Theorem~\ref{Thsistema'} is satisfied. Moreover, the alternatives 3 and 4 cannot  occur (see Remark \ref{obs}). Since (\ref{Ap2}) does not admit coexistence states for $\lambda$ small or large (according to Lemma \ref{NexistenciaAp2} (b) and (c)) and the coexistence states are bounded in $\mathcal{C}_0^1(\overline{\Omega})\times \mathcal{C}_0^1(\overline{\Omega})$ (see Lemma \ref{cotasAp2}), then  alternative 1 of Theorem~\ref{Thsistema'} cannot be satisfied either. Therefore, continuum $\mathfrak{C}$ satisfies alternative 2 of Theorem~\ref{Thsistema'}; that is, there exists a positive solution $(\lambda^*,\theta_{\lambda^*})$ such that $(\lambda^*,\theta_{\lambda^*},0)\in\overline{\mathfrak{C}} $ and
$$
\mu=\sigma_1[-\Delta+c \theta_{\lambda^*}]=\mu_{\lambda^*}.
$$
Hence, there exists coexistence state for $\lambda\in (\lambda_\mu,\lambda^*)$ (or  $\lambda\in (\lambda^*,\lambda_\mu)$).

Since $\mu> \lambda_1$ is arbitrary, we obtain the result.
\end{proof}

\begin{figure}
\centering
\includegraphics[scale=0.8]{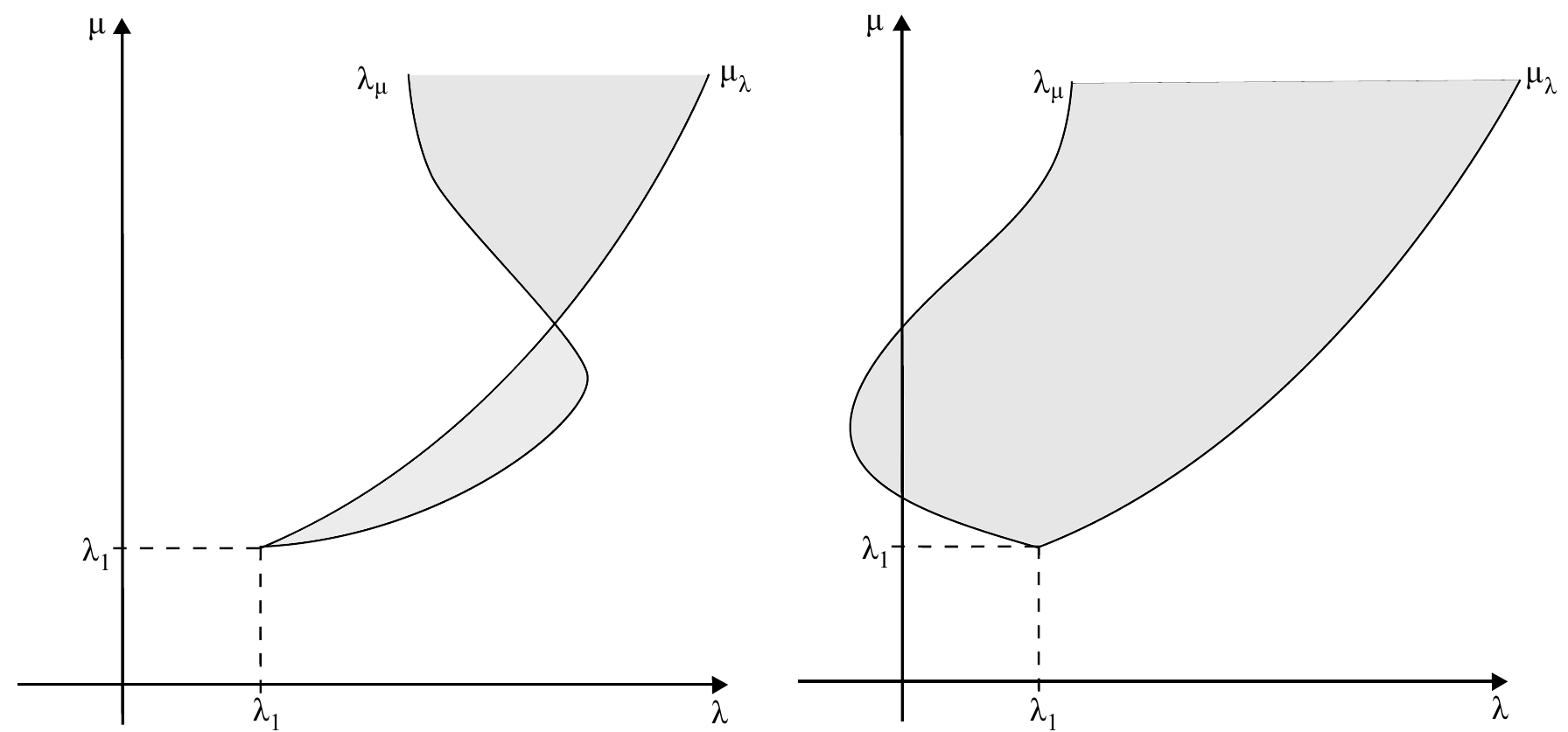}
\caption{\label{fig2}Admissible regions of coexistence of (\ref{Ap2}).}
\end{figure}

In Figure~\ref{fig2} we have represented some admissible region of coexistence of (\ref{Ap2}) given by Theorem \ref{thap2}.



\end{document}